\RequirePackage{fix-cm}
\documentclass[smallextended]{svjour3}       
\smartqed  
\usepackage{graphicx}
\usepackage{caption}
\usepackage{subcaption}
\usepackage{amsopn}

\usepackage{graphicx}
\usepackage{amsmath}
\usepackage{amsfonts}
\usepackage{amssymb}
\usepackage{multirow,booktabs,dcolumn}
\usepackage{import}
\usepackage{xcolor}
\usepackage{algorithm}
\usepackage{algorithmic}

\newcommand{\executeiffilenewer}[3]{%
	\ifnum\pdfstrcmp{\pdffilemoddate{#1}}%
	{\pdffilemoddate{#2}}>0%
	{\immediate\write18{#3}}\fi%
}

\graphicspath{{fig/}}
\let\theoremstyle\relax

\usepackage{amsthm}
\usepackage{mathtools}              
\newtheorem{mytheorem}{Theorem}
\newtheorem{mycorollary}{Corollary}
\newtheorem{mylemma}[theorem]{Lemma}
\theoremstyle{remark}
\newtheorem{myremark}{Remark}
\theoremstyle{definition}
\newtheorem{mydefinition}{Definition}

\begin{document}

\title{A gradient descent akin method for inequality constrained optimization
}


\author{Long Chen         \and
        Wenyi Chen \and
        Kai-Uwe Bletzinger 
}


\institute{L. Chen \at
              Technical University of Munich, Germany \\
              Tel.: +49 (89) 289 - 22462\\
              Fax: +49 (89) 289 - 22421\\
              \email{long.chen@tum.de}           
           \and
           W. Chen \at
              Wuhan Univeristy, China \\
			\email{wychencn@whu.edu.cn}    \\
           \and
			K.-U. Bletzinger \at
			Technical University of Munich, Germany \\
			\email{kub@tum.de}
}

\date{Submitted}

\maketitle

\begin{abstract}
We propose a first-order method for solving inequality constrained optimization problems. The method is derived from our previous work \cite{Chen}, a modified search direction method (MSDM) that applies the singular-value decomposition of normalized gradients. In this work, we simplify its computational framework to a ``gradient descent akin'' method, i.e., the search direction is computed using a linear combination of the negative and normalized objective and constraint gradient. The main focus of this work is to provide a mathematical aspect to the method. We analyze the global behavior and convergence of the method using a dynamical systems approach. We then prove that the resulting trajectories find local solutions by asymptotically converging to the central path(s) for the logarithmic barrier interior-point method under the so-called \textit{relative convex condition}. Numerical examples are reported, which include both common test examples and applications in shape optimization.

\end{abstract}

\section{Introduction}
\label{sec:intro}
In this paper, we propose a first-order method for solving inequality constrained optimization problems. The problem of interest is
\begin{equation}
	\begin{split}
		&\textnormal{minimize}~~~  f(x), \\
		&\textnormal{subject to}~~ g_i(x) \leq 0, ~ i = 1,...,m, \\
	\end{split}
	\label{eq:Optimization_Problem}
\end{equation}
where $f, g_1,..., g_m: \mathbb{R}^{n} \rightarrow \mathbb{R}$ are twice differentiable.

The present method is derived from our previous work \cite{Chen}, a modified search direction method (MSDM) for inequality constrained optimization problems. MSDM computes a descent direction of the objective function using the singular-value decomposition that exploits the normalized gradient information of a constrained optimization problem. Numerical experiments show that the MSDM finds local solutions by traversing along the central path(s) for the logarithmic barrier method. However, there is a lack of the mathematical theory of the method. The intrinsic optimization parameter in MSDM has remained heuristic. In this work, we simplify the computational framework of MSDM to a ``gradient descent akin" method, i.e., we compute the search direction using a linear combination of the negative and normalized objective and constraint gradient. We analyze the global behavior and convergence using a dynamical system perspective. We prove that the resulting trajectories find local solutions by asymptotically converging to the central path of the logarithmic barrier interior-point method. We should note that in this paper, we do not address the design and analysis of a practical implementation of the method, which is currently an active area of research.

\subsection{Main results}
\label{sec:results}
For problem (\ref{eq:Optimization_Problem}), we make the following assumptions:

\noindent (A1) Coercive condition for the objective function $f(x)$,

$$\lim_{x\rightarrow \infty}f(x)=+\infty;$$

\noindent (A2) \textcolor[rgb]{0,0,0}{$\nabla f(x) \neq 0$ in the feasible set $\Omega=\{x: g_i(x) \leq 0, ~ i = 1,...,m\}$;}

\noindent (A3) $\nabla g_i(x) \neq 0, ~ i = 1,...,m$ in the feasible set $\Omega$;

\noindent (A4) $f$ and $g_i$ are twice continuously differentiable functions.

Under assumptions (A1)$-$(A4), we propose a dynamical system for solving the problem (\ref{eq:Optimization_Problem}) that includes a single inequality constraint $g(x)$,
\begin{equation}
	\frac{d x}{d t} =-\frac{\nabla f}{|\nabla f|}
	-\zeta\frac{\nabla g}{|\nabla g|}, ~~ \zeta \in [0,1),
	\label{eq:modified_search_direction}
\end{equation}
where ``$|\cdot|$" denotes the Euclidean norm. $\nabla f$ and $\nabla g$ are the gradient row vectors of the objective function and constraint function, respectively.

We show analysis of the global behavior of the method. Among others, we propose an error measure $\epsilon$ for the optimization solution, and show that the time in which the present system finds a local solution is $\mathcal{O}(1/\epsilon)$, which is the ergodic convergence rate for first-order methods. We prove global and local convergence of the method. The present system results in a trajectory, which, under the so-called \textit{relative convex condition}, converges to the central path of the interior-point methods as $\zeta \rightarrow 1^-$.

Provided the well-known logarithmic barrier function
\begin{equation}
	\Phi(x) = - \sum_{i = 1}^{m} \log (-g_i(x)), ~ i = 1,...,m,
	\label{eq:log_barrier}
\end{equation}
where $\log(\cdot)$ denotes the natural logarithm, 
we propose a generalization of the system (\ref{eq:modified_search_direction}) for the multiple constrained problem with $\zeta\in [0,1)$ as
\begin{equation}
	\frac{d x}{d t}=\left\{
	\begin{split}
		&-\frac{\nabla f}{|\nabla f|}
		-\zeta\frac{\nabla \Phi}{|\nabla \Phi|},  ~~~~~&\text{if}~~\nabla \Phi \neq 0;\\
		& -\nabla f ,                              &\text{if}~~\nabla \Phi = 0.
	\end{split}\right.
	\label{eq: vector_s_multiple}
\end{equation}
Compared to our previous work \cite{Chen}, we simplified the framework to compute the search direction so that the implementation effort and computational cost are reduced significantly. We report computational experiences with various test examples that include both common benchmark problems and applications in shape optimization. Finally, we provide a large-scale real world shape optimization, which, to the best of our knowledge, is unlike any other case previously presented in the literature.

\subsection{Organization of paper}
In section \ref{sec:related}, we first review related works with the focus on theoretical aspects. In section \ref{sec:derivation}, we derive the method for single inequality constrained problems. Preliminary studies are then presented in section \ref{sec:preliminary}. Our main results are presented in section \ref{sec: global_behavior} - \ref{sec:multiple_constraint}. We report numerical examples in section \ref{sec:experiments}, and finally give a conclusion in section \ref{sec:conclusion}.

\section{Related works}
\label{sec:related}
We review related works to share our view of the present method from different well-established perspectives. Here, the emphasis is on the theoretical connections/differences as the main focus of this paper is to provide a theoretical foundation for the present method. We also briefly review the shape optimization as it has motivated the development of the method in this work.

\subsection{Gradient descent method}
The gradient descent method, originally proposed by Cauchy, is a first-order method for unconstrained optimization that uses the negative gradient of the objective function for the design update. In a dynamical system, it writes
$$ \frac{dx}{dt} = -\nabla f.$$
The direction of the gradient descent is the steepest descent direction in the Euclidean norm. A first-order Taylor expansion at current iterate $x$ of the objective function reads
\begin{equation*}
	f(x + d) \approx f(x) + \nabla f(x) d.
\end{equation*}
The steepest descent direction $d$ is found by the optimization problem
\begin{equation*}
	\begin{split}
		&\textnormal{minimize}~~~ \nabla f(x) d, \\
		&\textnormal{subject to}~~~ | d| = 1.
	\end{split}
\end{equation*}
From the Cauchy-Schwarz inequality, we obtain
\begin{equation}
	d = - \frac{\nabla f}{|\nabla f|}, 
	\label{eq:steepest_descent}
\end{equation}
which is the negative and normalized objective gradient.
Comparing (\ref{eq:steepest_descent}) and (\ref{eq:modified_search_direction}), which linearly combines the negative and normalized objective and constraint gradient, we consider the present system: a gradient descent akin method for inequality constrained optimization problems. We especially note that the present method results in trajectories that are homotopic with the gradient descent trajectory (see Remark \ref{remark1}).

\subsection{Dynamical systems approaches} 
Dynamical systems approaches have been used to study optimization methods in many works of literature. Extensive studies on the connections between interior-point flows with linear programming methods can be found in \cite[Chapter~4]{helmke2012optimization} and the references therein. For quadratic programming problems, \cite{dorr2012smooth} proposes a dynamical system, which results in trajectories that converge to the saddle point of the associated Lagrangian function. \cite{su2014differential} studies the celebrated Nesterov's accelerated gradient method using a dynamical system as the analysis tool. In \cite{lessard2016analysis}, a framework based on dynamical systems is proposed to analyze and design first-order unconstrained optimization methods. In this work, we study the behavior of the present method using the dynamical systems approach. The ODE interpretation of the method allows us to give a rigorous analysis of its global and local convergence. 

In some literature, optimization methods that use dynamical systems are called trajectory methods. These methods construct optimization paths in a way so that one or all solutions to the optimization problem are \textit{a priori} known to lie on these paths \cite{diener1995trajectory}. Typically, these optimization paths are solution trajectories to ODE of first or second-order. Trajectory methods are mainly studied for unconstrained optimizations for finding local solutions \cite{behrman1998efficient}\cite{botsaris1978differential}, and global solutions \cite{griewank1981generalized}\cite{snyman1987multi}. Studies for constrained optimization are, however, very limited, see \cite{ali2018trajectory} and the references therein. In this work, we propose a new dynamical system for inequality constrained optimization problems. The resulting trajectories find local solutions by the limiting behavior as $\zeta \rightarrow 1^-$.

\subsection{Interior-point methods}

Interior-point methods (IPMs), which are based on the Newton method, are among the most competitive methods for constrained optimization problems. The signature of IPM is the existence of continuously parameterized families of approximate solutions that converge to the exact solution asymptotically \cite{Forsgren}. IPMs find a wide variety of applications of convex and nonconvex optimizations in broad fields. There are a vast amount of excellent works that have been devoted to IPM. A comprehensive review of this class of methods is certainly beyond the scope of the present paper, however we refer the interested reader to \cite{Forsgren}\cite{Potra}, more recently, in \cite{gondzio2012interior}, and many other excellent optimization books.  

The present method results in trajectories that asymptotically converge to the central path for a particular IPM, the logarithmic barrier method \cite{fiacco1990nonlinear}. We introduce its connections/differences to the present method in the next. Consider a convex optimization problem of the form (\ref{eq:Optimization_Problem}), we start with its approximated unconstrained problem using the logarithmic barrier function $\Phi$,
\begin{equation}
	\textnormal{minimize} ~~~~  f(x) + \eta \Phi(x),
	\label{eq:logarithmic_barrier}
\end{equation}
where the barrier parameter $\eta$ is a positive parameter. As $\eta \rightarrow 0$, the solution of the approximated problem converges to the original one. The central path is characterized by the set of points that satisfy the necessary and sufficient conditions \cite{boyd}: 
\begin{equation}
	0 = \nabla f(x^*) + \eta \nabla \Phi(x^*), ~~x^* \in \Omega_-,
	\label{eq:barrier_central_point}
\end{equation}
where $\Omega_-=\{x: g_i(x) < 0, ~ i = 1,...,m\}$. The conditions (\ref{eq:barrier_central_point}) are interpreted as a modified KKT system in the literature \cite{boyd}\cite{byrd2000trust}\cite{Forsgren}. The barrier method finds an approximated solution for the original problem by 1) iteratively decreasing the barrier parameter $\eta$, and 2) in each iteration, solving the subproblem (the modified KKT system) defined by (\ref{eq:barrier_central_point}) using the Newton method. 
Therefore, the barrier parameter $\eta$ can be considered a central path parameter.

In the present method, we do not parameterize the central path. Instead, we normalized the objective and constraint gradients. For optimization problems with a single inequality constraint, we propose the \textit{normalized central path condition} as
\begin{equation}
	\frac{\nabla f}{|\nabla f|} + \frac{\nabla g}{|\nabla g|}  = 0.
	\label{eq:normalized_central_path_condition_single_constraint}
\end{equation}

The condition (\ref{eq:normalized_central_path_condition_single_constraint}) is used in section \ref{sec: global_behavior} and \ref{sec:local_analysis} for the analysis. A generalization of the \textit{normalized central path condition} from single inequality to multiple inequalities is introduced by the use of the logarithmic barrier function (see section \ref{sec:multiple_constraint}) as
\begin{equation}
	\frac{\nabla f}{|\nabla f|} + \frac{\nabla \Phi}{|\nabla \Phi|}  = 0.
	\label{eq:normalized_central_path_condition}
\end{equation}
Compared with the barrier method, the path parameter $\eta$ has vanished. The condition (\ref{eq:normalized_central_path_condition}) characterizes the central path, which differs from (\ref{eq:barrier_central_point}), which instead characterizes a point on the central path.

To illustrate the behavior of the present method in a very rough way, we rewrite the first ODE of system (\ref{eq: vector_s_multiple}) as 
\begin{equation}
	\frac{d x}{d t} =-\frac{\nabla f}{|\nabla f|}(1- \zeta) 
	+ \left( -\frac{\nabla f}{|\nabla f|} - \frac{\nabla \Phi}{|\nabla \Phi|}\right) \zeta, ~~ \zeta \in [0,1).
	\label{eq:homotopy}
\end{equation}
While $-\frac{\nabla f}{|\nabla f|} $ is the steepest descent direction of the objective function, the term $\left( -\frac{\nabla f}{|\nabla f|} - \frac{\nabla \Phi}{|\nabla \Phi|}\right)$ contributes to the centering behavior (referred to Theorem \ref{theorem3}). There are three major differences compared to the barrier method:
\begin{itemize}
	\item[1)] While $\eta$ is a path parameter that controls the asymptotic convergence progress for the barrier method, $\zeta$ is a homotopy parameter that determines the shape of the optimization trajectory in the present method;
	\item[2)] There is no subproblem (modified KKT system) defined and solved in the present method. Instead, we solve directly for the trajectory of the proposed dynamical system. The optimization solutions are known \textit{a priori} to lie on the resulting trajectory;
	\item[3)] The local and global convergence behavior of the present method are markedly different from the barrier method.
\end{itemize}

The asymptotic convergence for the barrier method along the central path has been extensively studied in {\cite{fiacco1990nonlinear}}. For general inequality-constrained problems (\ref{eq:Optimization_Problem}) under mild assumptions (A1)-(A4), to force global convergence, IPM usually implements a merit function associated line-search method or trust region framework \cite{byrd2000trust}\cite{vanderbei1999interior}\cite{wachter2006implementation}. To provide a global and local convergence theory for the present method, we prove the following:
\begin{itemize}
	\item[-] The resulting trajectory converges to a critical point of the objective function as $\zeta \in [0,1)$ (see Theorem \ref{theorem2});
	\item[-] The trajectory finds a KKT solution upon reaching the boundary of the feasible set $\Omega$ as $\zeta \rightarrow 1^-$, provided that there is no critical point of the objective function in $\Omega$ (see Theorem \ref{theorem1}, \ref{theorem5}, and Remark \ref{remark_fo});
	\item[-] As $\zeta \rightarrow 1^-$, the second-order optimality conditions are automatically satisfied at KKT solutions (see section \ref{sec:local_analysis});
	\item[-] The trajectory is able to switch between central paths to remain a descent direction of the objective function without using an additional framework (referred to Theorem \ref{theorem8} and Lemma \ref{lemma2}).
\end{itemize}

From a more abstract point of view, the difference between the present method and the barrier method may be analogous to the difference between the gradient descent method and the Newton-based method for unconstrained minimizations. While the former seek critical points of the objective function or KKT solutions that satisfy the respective second-order optimality conditions using the negative gradient information, the latter seek the respective solutions using the Newton method.

\subsection{Feasible direction methods}
The method of feasible directions (MFD) dates back to the 1960's by the work of Zoutendijk \cite{zoutendijk1960methods} and has enjoyed fruitful developments for decades. MFDs have been especially popular in the engineering community because of the importance of ending up with a design that satisfies the hard specifications expressed by a set of inequalities \cite{chen2000methods}. The general idea behind the MFD is to move from one feasible design to an improved feasible design iteratively so that a local solution can be found \cite{Arora}. In the present method, remaining feasibility is not a mandatory mechanism. The idea behind the method design is to find a search direction that approaches the central path while maintaining a descent direction of the objective function. Due to this major difference, we do not categorize our method as a method of feasible directions.

In the present method, we compute a search direction that uses normalized gradients. In \cite{stander1993new}, the authors also present a feasible direction method that applies normalized gradients. Their work is then continued and further developed in \cite{de1994feasible} and \cite{stander1995robustness}. In these works, the active-set strategy is used. The common idea is to formulate a linear system under given input criteria on a chosen working set of (active) constraints, and a feasible descent search direction is obtained by solving the linear system. In the present method, we use a barrier function based formulation to treat multiple inequality constraints. The search direction is computed as a linear combination of the negative and normalized objective and barrier gradient.

\subsection{Shape optimization}
As a subset of structural optimization, shape optimization is characterized by a very large or even infinite number of design variables that describe the varying boundary in the optimization process. Introductions to shape optimization are given in \cite{haslinger2003introduction}\cite{sokolowski1992introduction}. Shape optimization is distinct from another well-known problem in structural engineering: topology optimization \cite{bendsoe2013topology} (sometimes referred to as the homogenization method \cite{allaire2012shape}). The main difference is that the topology optimization method removes smoothness and topological constraints in shape optimization, which results in different optimization formulations. Many topology optimization problems can be formulated in an (equivalent) convex optimization problem, while shape optimizations are typically nonlinear, and
often nonconvex \cite{hoppe2007adaptive}. This difference partially contributes to the fact that there are successful implementations of IPM for large-scale topology optimization problems \cite{jarre1998optimal}\cite{kocvara2016primal}\cite{maar2000interior}, but only a few works have presented a shape optimization that uses an IPM as the optimizer \cite{antil2007path}\cite{herskovits2000shape}. In the latter works, the size of the shape optimization problem is only moderate so that the power of IPM is not fully exploited. One of the most successful methods for nonlinear topology optimization is the method of moving asymptotes (MMA) that was introduced by Svanberg in 1987 \cite{svanberg1987method}. In each iteration, MMA generates and solves an approximated convex problem related to the original one. For shape optimization, however, there is as yet no literature that discusses a large-scale problem using MMA.

A notable difficulty for shape optimization is the computation of shape Hessians, which are complex objects even for moderate problems. Analysis of aerodynamic optimization in \cite{arian1999analysis} shows that shape Hessians are ill-conditioned for three-dimensional problems. Recently, several works compute approximated shape Hessians and use a Newton-based method for the design optimization \cite{schillings2011efficient}\cite{schmidt2013three}. For some disciplines, such as computational fluid dynamics or transient coupled problems, even the computation of shape gradient can be a challenge. See for example \cite{albring2016efficient}\cite{korelc2009automation}\cite{reuther1999constrained}, which are actively undergoing investigation.

In general, large-scale shape optimization is mainly performed using gradient descent type method so far \cite{schulz2015}. In engineering practice, a large number of constraints may be considered. The lack of literature in this regard has motivated our development of MSDM in \cite{Chen}. In the present work, we further simplify the computational framework of MSDM to a gradient descent akin method. We provide a mathematical basis for the method's optimization behavior. As a result, the implementation effort and computational cost are reduced significantly. It opens the possibility of shape optimization to a wider range of applications.

\section{Deriving the search direction for single inequality constrained optimizations}
\label{sec:derivation}
In this section, we show the consistent derivation of the system (\ref{eq:modified_search_direction}) from our previous work \cite{Chen} considering a single inequality constraint. First, we review the basic ideas of the modified search direction method and then show the derivation.

In MSDM, at each iteration, we construct a sensitivity matrix
\begin{equation}
	\mathbf{m} = \begin{pmatrix}  \frac{\nabla f}{|\nabla f|} \\ \frac{\nabla g}{|\nabla g|} \end{pmatrix}.
	\label{eq: sensitivity_matrix}
\end{equation}
The change in the objective function $df$ and the constraint function $dg$ resulting from an arbitrary design change $dx  \in \mathbb{R}^n$ reads
\begin{equation}
	\begin{split}
		df = \nabla f dx, \\
		dg = \nabla g dx.
	\end{split}
\end{equation}
A perspective from an input-output system established by the sensitivity matrix $\mathbf{m}$ gives
\begin{equation}
	\begin{pmatrix}  \frac{df}{|\nabla f|} \\\frac{dg}{|\nabla g|} \end{pmatrix} = \mathbf{m} dx.
	\label{eq:input_output_smatrix}
\end{equation}
Applying singular-value decomposition to the sensitivity matrix $\mathbf{m}$,
\begin{equation}
	\mathbf{m} = \mathbf{U} \mathbf{\Sigma} \mathbf{V}^T = \sum_{i = 1}^{min(2,n)} \mathbf{\sigma}_{i} \mathbf{u}_i \mathbf{v}_i^T.
	\label{eq: svd_m}
\end{equation}
Thus, an orthonormal bases set $\mathbf{v}_i, ~i = 1,2$ is obtained. Each $\mathbf{v}_i$ can be used as a base search direction for the design update, and $\mathbf{v}_1$ and $\mathbf{v}_2$ are defined as follows:
\begin{itemize}
	\item[-] $\mathbf{v}_1$: by taking $\delta \mathbf{v}_{1}$ as the design change, we obtain a change in objective as well as in constraint function $[\frac{df}{|\nabla f|}, \frac{dg}{|\nabla g|}]^T = \sigma_{1} \delta \mathbf{u}_{1} $, which is a \textbf{decrease} in the objective function and an \textbf{increase} in the constraint function.
	
	\item[-] $\mathbf{v}_2$: by taking $\delta \mathbf{v}_{2}$ as the design change, we obtain a change in the objective as well as in the constraint function $[\frac{df}{|\nabla f|}, \frac{dg}{|\nabla g|}]^T = \sigma_{2} \delta \mathbf{u}_{2} $, which is a \textbf{decrease} in the objective function and a \textbf{decrease} in the constraint function. 	
\end{itemize}
It is worth mentioning that the design vector $\delta \mathbf{v}_1$ provides a similar result as the filter approach presented in \cite{fletcher2002nonlinear}, which tries to minimize the so-called bi-objective optimization problem with two goals of minimizing the objective function $f$ and the constraint violation $|g|$ (with the difference that $g$ takes different signs in both cases).

With $\mathbf{v}_1$ and $\mathbf{v}_2$, we can then rewrite the normalized steepest descent direction $-\frac{\nabla f}{|\nabla f|}$ as
\begin{equation}
	-\frac{\nabla f}{|\nabla f|} = \cos \alpha_1 \mathbf{v}^T_{1} + \cos \alpha_2 \mathbf{v}^T_{2},
	\label{eq:rewrite_steepest_descent}
\end{equation}
where $\alpha_1$ is the angle between $\mathbf{v}^T_{1}$ and $-\frac{\nabla f}{|\nabla f|}$, and $\alpha_2$ is the angle between $\mathbf{v}^T_{2}$ and $-\frac{\nabla f}{|\nabla f|}$. The modified search direction proposed in \cite{Chen} reads
\begin{equation}
	\mathbf{s}_c = \cos \alpha_1 \mathbf{v}^T_{1} + c \cdot \cos \alpha_2 \mathbf{v}^T_{2},
	\label{eq:modified_search_direction_svd}
\end{equation}
where $c \geq 1$ is introduced to enlarge the contribution of the design mode $\mathbf{v}_2$.

In the following, we show the derivation of (\ref{eq:modified_search_direction}) from (\ref{eq:modified_search_direction_svd}). According to SVD and the definition of $\mathbf{v}_1$ and $\mathbf{v}_2$, we have
\begin{equation}
	\begin{split}
		\mathbf{v}^T_1&=\frac{1}{\sqrt{2-2\cos\theta}}\left( -\frac{\nabla f}{|\nabla f|}+\frac{\nabla g}{|\nabla g|}\right),
		\\
		\mathbf{v}^T_2&=\frac{1}{\sqrt{2+2\cos\theta}} \left(-\frac{\nabla f}{|\nabla f|}-\frac{\nabla g}{|\nabla g|}\right),
	\end{split}
	\label{eq:base_vectors}
\end{equation}
where $\theta$ is the angle between the objective function gradient $\nabla f$ and the constraint function gradient $\nabla g$. With $\theta$ we also have
\begin{equation}
	\begin{split}
		\cos \alpha_1=-<\frac{\nabla f}{|\nabla f|},\mathbf{v}^T_1>=\frac{\sqrt{1-\cos\theta}}{\sqrt{2}},\\ \cos \alpha_2=-<\frac{\nabla f}{|\nabla f|},\mathbf{v}^T_2>=\frac{\sqrt{1+\cos\theta}}{\sqrt{2}}.
	\end{split}
	\label{eq:cos_alphas}
\end{equation}
Inserting (\ref{eq:base_vectors}), (\ref{eq:cos_alphas}) into (\ref{eq:modified_search_direction_svd}) we have
\begin{equation}
	\mathbf{s}_c= -\frac{\nabla f}{|\nabla f|}-\frac{(c-1)}{2} \left(\frac{\nabla f}{|\nabla f|}+\frac{\nabla g}{|\nabla g|}\right).
\end{equation}
As we are mainly interested in the direction of the vector field $\mathbf{s}_c$, we can rewrite it as
\begin{equation}
	\mathbf{s}_\zeta =-\frac{\nabla f}{|\nabla f|}
	-\zeta\frac{\nabla g}{|\nabla g|},
	\label{vector_s}
\end{equation}
with $\zeta = \frac{c-1}{c+1}$. With $c \in [1,+\infty)$ we have $\zeta \in [0, 1)$ and thus we get the present dynamical system (\ref{eq:modified_search_direction}).

\section{Preliminary studies and intuition}
\label{sec:preliminary}
In this section, we demonstrate the present method on two simple examples and conjecture the behavior of the resulting optimization trajectory with intuition.
\subsection{An analytical 2D optimization example}
We first show a 2D optimization problem and solve it analytically. The optimization problem reads,
\begin{equation}
	\begin{split}
		&\textnormal{minimize}~~~	f(x_1,x_2)=\frac{1}{2}(x_1^2 + x_2^2), \\
		&\textnormal{subject to}~~  g(x_1,x_2)= - x_2 + 10 \leq 0. \\
	\end{split}
	\label{eq:example_proof}
\end{equation}

\begin{figure}[h]
	\centering
	\begin{footnotesize}
	\executeiffilenewer{fig/2d_linear_constraint_analysis_zeta.svg}{fig/2d_linear_constraint_analysis_zeta.pdf}%
	{inkscape -z -C --file=fig/2d_linear_constraint_analysis_zeta.svg %
		--export-pdf=fig/2d_linear_constraint_analysis_zeta.pdf --export-latex}%
	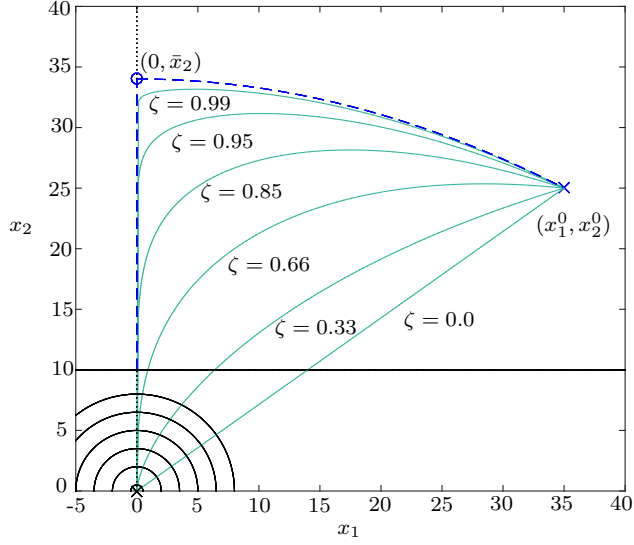%

		\caption{Optimization trajectories for the 2D linear constrained optimization problem \ref{eq:example_proof} with different parameter $\zeta$. The black circles show the objective function contours. The black line shows the constraint function. The dotted line is the central path. As $\zeta \rightarrow 1^{-}$, a part of the trajectory $\Gamma^\zeta$ based on the present search direction field (\ref{vector_s}) converges to the central path.}
		\label{fig:2d_linear_constraint_analysis}
	\end{footnotesize}
\end{figure}

The present search direction field $\mathbf{s}_\zeta$ reads
\begin{equation}
	\mathbf{s}_\zeta = \frac{-1}{\sqrt{x_1^2 + x_2^2}}\left\lbrace  x_1,  x_2 - \zeta \sqrt{x_1^2 + x_2^2} \right\rbrace.
\end{equation}
We define an initial design as $(x_1^0, x_2^0)$. Let $\bar{x}_2 = \frac{1}{2} \left(x_2^0 + \sqrt{(x_1^0)^2 + (x_2^0)^2}\right)$, then, the trajectory $\Gamma^\zeta$ of the present search direction field $\mathbf{s}_\zeta$ is
\begin{equation}
	x_2 + \sqrt{x_1^2 + x_2^2} = 2 \bar{x}_2 \left|\frac{x_1}{x_1^0}\right|^{1- \zeta}.
\end{equation}
Let  $(x_{1,\zeta}, x_{2,\zeta})$ be a point on the trajectory $\Gamma^\zeta$ with a maximal $x_2$ component, then we have
\begin{equation}
	(x_{2,\zeta})^\zeta=\frac{2}{1+\zeta}\frac{\bar{x}_2}{|x_1^0|^{1-\zeta}}
	\left(\frac{\sqrt{1-\zeta^2}}{\zeta}\right)^{1-\zeta},
\end{equation}
and
\begin{equation}
	|x_{1,\zeta}|=\frac{\bar{x}_2}{\zeta}\sqrt{1-\zeta^2}. 
\end{equation}
Let $\zeta \rightarrow 1^{-} $, then, $x_{1,\zeta} \rightarrow 0$, $x_{2,\zeta} \rightarrow \bar{x}_2$, the trajectory $\Gamma^\zeta$ will converge to the curve
$\Gamma$ that is a union of the parabola
\begin{equation}
	x_1^2 =  4 \bar{x}_2^2 - 4 x_2 \bar{x}_2, ~~~~ x_1\in (0, x_1^0) (~~~\mbox{or} (x_1^0,0),)
\end{equation}
and the interval $(0,\bar{x}_2)$ on $x_2$-axis as is shown in figure \ref{fig:2d_linear_constraint_analysis}.
As $\zeta \rightarrow 1^{-} $, a part of the trajectory $\Gamma^\zeta$ converges to the central path. For $\zeta \in [0,1)$, the resulting trajectories are homotopic relative to their endpoints, which are the initial design and the critical point of the objective function.

By $\bar{x}_2 = \frac{1}{2} \left(x_2^0 + \sqrt{(x_1^0)^2 + (x_2^0)^2}\right)$, we have $\bar{x}_2 \geq x_2^0$. This means for any feasible initial design, as $\zeta \rightarrow 1^-$, the resulting optimization trajectory always reaches first a close neighborhood of the central path (at point $(0, \bar{x}_2)$) , where it is at a larger distance to the boundary of the feasible set compared to the initial design. It then follows the central path and reaches the optimal solution.

\subsection{A nonconvex optimization example} We show a second example that includes a nonconvex constraint,
\begin{equation}
	\begin{split}
		&\textnormal{minimize}~~~	f(x_1,x_2)= (x_1 - 2)^2 + (x_2 -2)^2, \\
		&\textnormal{subject to}~~  g(x_1,x_2)= -\frac{1}{10}(x_1 - 3)^2 - x_2 + 3 \leq 0. \\
	\end{split}
	\label{eq:2d_quadrati_obj_nonconvex_constraint}
\end{equation}
In figure \ref{fig:NGopt_non_convex_contours_new}, we plot the optimization trajectory with $\zeta = 0.9999$ together with a few depicted contours of both the objective function and the constraint function. We plot three points A, B, and C on the central path, where the plotted contours of the objective and constraint function are tangent to each other. We choose an initial design $\mathbf{x}^0$ that is close to the point A. It can be observed: instead of heading to the left side of the central path, the optimization trajectory finds its way to the right side. It then follows the central path, but leaves at point C and reaches the other central path. It eventually finds the optimal solution by following this central path. Along the optimization trajectory, the objective function value decreases steadily. The question is now, why does the optimization trajectory choose one side (point B side) of the same central path over another (point A side)? The answer may lie in the difference in the curvatures of the contours of the objective and constraint function between point A and B. We observe the following fact:

Let $\kappa_f$ and $\kappa_g$ be the curvature of the contours of objective and constraint function at the central path, respectively, both with the normal vector $\frac{\nabla f}{|\nabla f|}$, then
\begin{itemize}
	\item[-] at point A: $\kappa_f - \kappa_g > 0$;
	\item[-] at point B: $\kappa_f - \kappa_g < 0$.	
\end{itemize}

\begin{figure}[h]
	\centering
	\begin{footnotesize}
	\executeiffilenewer{fig/NGopt_non_convex_contours_red.svg}{fig/NGopt_non_convex_contours_red.pdf}%
	{inkscape -z -C --file=fig/NGopt_non_convex_contours_red.svg %
		--export-pdf=fig/NGopt_non_convex_contours_red.pdf --export-latex}%
	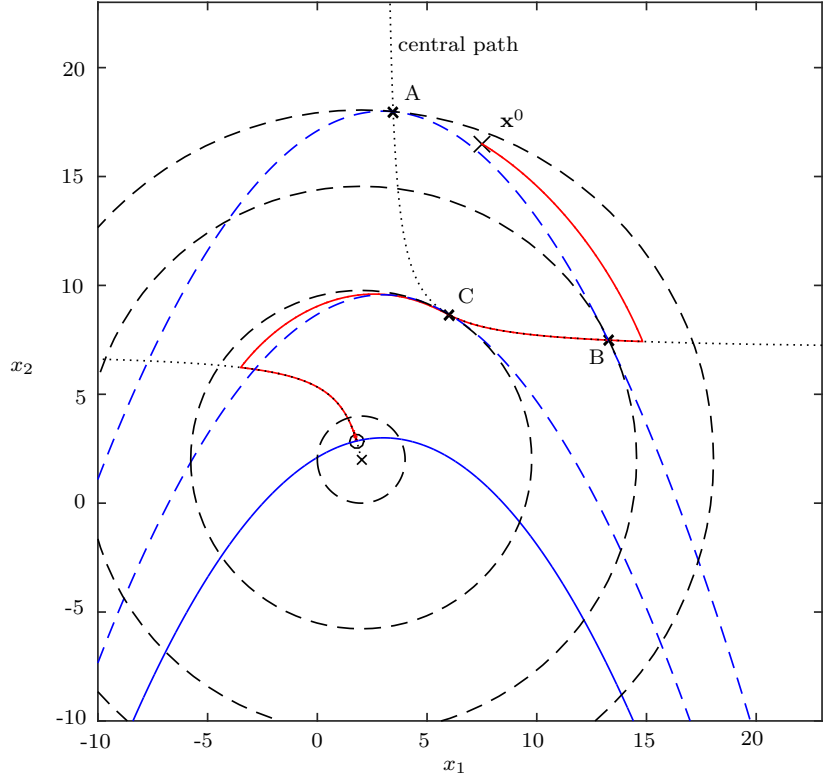%

		\caption{A study on the behavior of the optimization trajectory for problem (\ref{eq:2d_quadrati_obj_nonconvex_constraint}). Red line is the optimization trajectory. Black dashed circles are the contours of the objective function, while the blue dashed curves show the contours of the constraint function. The dotted black lines are the central paths.}
		\label{fig:NGopt_non_convex_contours_new}
	\end{footnotesize}
\end{figure}

Based on this observation, we conjecture the behavior of the optimization trajectory: As $\zeta \rightarrow 1^-$, the optimization trajectory is able to approach and follow a central path, on which the central point satisfies the condition
\begin{equation}
	\kappa_f - \kappa_g < 0. 
	\label{eq: relative_convex}
\end{equation}
We call condition (\ref{eq: relative_convex}) \textit{the relative convex condition}. This behavior can also be used to explain why the optimization trajectory leaves the central path at point C, where  $\kappa_f = \kappa_g$, and heads towards another central path. In section \ref{sec:local_analysis}, we give a local convergence analysis starting with this conjecture.

\section{Global behavior and convergence under coercive condition}
\label{sec: global_behavior}
We show analysis for the global behavior for the present dynamical system (\ref{eq:modified_search_direction}) under assumptions (A1)$-$(A4). To this purpose, we partly use the geometric analysis in \cite{jost2011riemannian}.
%
%
%
%
%

Consider a $\mu-$neighborhood of a central path with $\mu \in [0,1]$,
$$\Theta_\mu=\{x: g(x) \leq 0, \cos\theta<-\mu\}.$$
Obviously,  $\Theta_\mu$ shrinks to the central path as $\mu \rightarrow 1^-.$

Recall the present search direction field $\textbf{s}_\zeta$,
\begin{equation*}
	\textbf{s}_\zeta =-\frac{\nabla f}{|\nabla f|}
	-\zeta\frac{\nabla g}{|\nabla g|}.
	\label{vector-s}
\end{equation*}
Let $x(t; \zeta,x_0)$ be the solution of the system:

\begin{equation}\label{system}
	\left\{
	\begin{split}
		&\frac{dx}{dt}= \textbf{s}_\zeta(x),\\
		& x|_{t=0} = x_0.
	\end{split}\right.
\end{equation}
with $x_0$ as the initial design in the feasible set and $T_{\zeta,x_0}$ as the maximal existence interval of $x(t; \zeta,x_0)$ in the whole $\mathbb{R}^n$.  We show some properties of  the trajectory  $x(t; \zeta,x_0)$.
\vskip 2mm
\begin{mylemma}
	Let $f$ and $g$ satisfy  assumptions (A1)(A3) and (A4), then
	
	(i) The trajectory $x(t; \zeta,x_0)$ always stays within a bounded domain, i.e.,
	$$x(t; \zeta,x_0)\in \Omega_{f(x_0)}=\{x: f(x) \leq f(x_0)\}.$$
	
	(ii) The constraint function $g$ decreases along the trajectory $x(t; \zeta,x_0)$ out of the cone neighborhood $\Theta_{\zeta}$, and increases in $\Theta_{ \zeta}$.
	\label{lemma1}	
\end{mylemma}

\begin{proof}
	~~Based on the deformation of the objective function $f$ and constraint function $g$ along the trajectory $x(t; \zeta,x_0)$:
	
	\begin{equation}\label{deform}
		\left\{
		\begin{split}&
			\frac{d }{dt}f(x(t; \zeta,x_0))= -|\nabla f|(1+\zeta \cos \theta),
			\\ &
			\frac{d }{dt}g(x(t; \zeta,x_0))= -|\nabla g|(\zeta +\cos \theta).
		\end{split}
		\right.
	\end{equation}
	We get a proof directly.
\end{proof}

\begin{mytheorem}
	Suppose that assumptions (A1) - (A4) hold. Then
	
	(i) the trajectory $x(t; \zeta,x_0)$ must go out of the feasible set with $\zeta \in [0,1)$;
	
	(ii) the minimum time in which the trajectory reaches the boundary of the feasible set is at most $\frac{C}{1-\zeta}$ with $C$ independent of $\zeta$.
	\label{theorem1}
\end{mytheorem}

\begin{proof}
	\textcolor[rgb]{0,0,0}{~~We prove (i) by contradiction. Suppose that $x(t; \zeta,x_0)$ stays in the feasible set $\Omega $ for any $t<T_{\zeta, x_0}$, then the vector field $\textbf{s}_{\zeta}$ keeps $C^1$ continuous in a neighborhood of trajectory $x(t; \zeta,x_0)|_{[0,T_{\zeta, x_0})}$. Lemma 1 ensures that }
	\begin{equation}
		|\nabla f(x(t; \zeta,x_0))|\geq A, |f(x(t; \zeta,x_0)|\leq B, ~~~ \forall t<T_{\zeta, x_0},
		\label{lowerbound}
	\end{equation}
	with some positive numbers $A$ and $B$. On the other hand, Picard's existence theorem implies that $T_{\zeta, x_0}=+\infty.$
	The integral of the first formula of (\ref{deform}) shows
	$$
	\int_0^\infty|\nabla f(x(t; \zeta,x_0))|(1+\zeta \cos \theta)dt<\infty.
	$$
	This is to say
	\begin{equation}
		\int_0^\infty|\nabla f(x(t; \zeta,x_0))|dt<\infty.
		\label{finit integral}
	\end{equation}
	
	Notice that
	\begin{equation*}
		\begin{split}
			&\frac{d}{dt}|\nabla f(x(t; \zeta,x_0))|=\nabla |\nabla f(x(t; \zeta,x_0))|\cdot \frac{dx}{dt}\\
			& \hskip 3mm=\nabla |\nabla f(x(t; \zeta,x_0))|\cdot \textbf{s}_\zeta(x)\\
			& \hskip 3mm=\frac{-1}{|\nabla f|}\sum\limits_{k=1,j=1}^{n}\frac{\partial f}{\partial x_k } \frac{\partial^2 f}{\partial x_k \partial x_j} \left(\frac{1}{|\nabla f|}\frac{\partial f}{\partial x_j }+\frac{\zeta}{|\nabla g|}\frac{\partial g}{\partial x_j }\right).
		\end{split}
	\end{equation*}
	\textcolor[rgb]{0,0,0}{ 
		So
		\begin{equation*}
			\left|\frac{d}{dt}|\nabla f(x(t; \zeta,x_0))|\right|\leq 2 \sqrt{\sum\limits_{k=1,j=1}^{n} \left|\frac{\partial^2 f}{\partial x_k \partial x_j} \right|^2}.
	\end{equation*}}
	By assumption (A4) and Lemma \ref{lemma1}, there is a constant $l$ so that
	\begin{equation*}
		\left|\frac{d}{dt}|\nabla f(x(t; \zeta,x_0))|\right|\leq l, ~\forall t.
	\end{equation*}
	Hence, we have a Lipschitz continuity for $|\nabla f(x(t; \zeta,x_0))|$:
	\begin{equation}
		\left||\nabla f(x(t^\prime; \zeta,x_0))|-|\nabla f(x(t^{\prime\prime}; \zeta,x_0))|\right| \leq l |t^\prime-t^{\prime\prime}|,~~~~~ \forall t^\prime,t^{\prime\prime}.
		\label{lipschitz}
	\end{equation}
	
	Now, we claim
	\begin{equation}
		\lim_{t\rightarrow +\infty}|\nabla f(x(t;\zeta,x_0))|=0.
		\label{claim}
	\end{equation}
	Otherwise, there is a sequence of $t_j\rightarrow +\infty$ and a positive constant $b$ so that
	\begin{equation*}
		|\nabla f(x(t_j;\zeta,x_0))|\geq b>0.
	\end{equation*}
	Choosing $\delta=\frac{b}{2l}$, then
	\begin{equation*}
		\begin{split}
			&|\nabla f(x(t; \zeta,x_0))|\geq |\nabla f(x(t_j; \zeta,x_0))|\\
			&\hskip 3mm -\left||\nabla f(x(t_j; \zeta,x_0))|-|\nabla f(x(t; \zeta,x_0))|\right| \\
			& \hskip 3mm \geq |\nabla f(x(t_j; \zeta,x_0))|-l|t_j-t|\geq b-\delta l=\frac{b}{2}
		\end{split}
	\end{equation*}
	for any $|t_j-t|\leq \delta.$ Therefore,
	\begin{equation*}
		\begin{split}
			&\int_0^\infty|\nabla f(x(t; \zeta,x_0))|dt\geq \sum\limits_{j=1}^{\infty} \int_{t_j-\delta}^{t_j+\delta}|\nabla f(x(t; \zeta,x_0))|dt\\
			&\hskip 3mm \geq \sum\limits_{j=1}^{\infty} \int_{t_j-\delta}^{t_j+\delta}\frac{b}{2}dt=\sum\limits_{j=1}^{\infty} \delta b=+\infty.
		\end{split}
	\end{equation*}
	\textcolor[rgb]{0,0,0}{ This is a contradiction to (\ref{finit integral}). Hence 
		\begin{equation*}
			\lim_{t\rightarrow +\infty}|\nabla f(x(t;\zeta,x_0))|=0.
			\label{claim}
		\end{equation*}
		Notice that this convergence contradicts (\ref{lowerbound}).
		This completes the proof for (i).}
	\vskip 2mm
	
	To prove (ii),  let $T_{\zeta, x_0}^\sharp$ be the time in which the trajectory first reaches the  boundary of feasible set, then (\ref{lowerbound}) holds for $0<t<T_{\zeta, x_0}^\sharp$. So
	$$
	\int_0^{T_{\zeta, x_0}^\sharp}|\nabla f|(1+\zeta \cos \theta)dt=f(x_0)-f(x(T_{\zeta, x_0}^\sharp; \zeta,x_0))\leq 2B.
	$$
	Hence
	$$
	T_{\zeta, x_0}^\sharp A(1-\zeta)\leq 2B,
	$$
	which implies
	$$
	T_{\zeta, x_0}^\sharp\leq \frac{2B}{A(1-\zeta)},
	$$
	so that $T_{\zeta, x_0}^\sharp \leq \frac{C}{1-\zeta}$ holds. This ends the proof.
\end{proof}

\begin{mytheorem}
	Suppose that assumptions (A1),(A3), (A4) hold and $\nabla g\neq 0$ in the whole $\mathbb{R}^n$. The trajectory $x(t; \zeta,x_0)$ converges to a connected subset of  critical points of the objective function by $\zeta \in [0,1)$. Especially, if the critical points of objective function is isolated, then
	\begin{equation}\label{convergencetocritical}
		\lim_{t\rightarrow T_{\zeta,x_0}^-}x(t; \zeta,x_0)=x_c,~~~~~~~~~~ \forall \zeta \in [0,1).
	\end{equation}
	\label{theorem2}
\end{mytheorem}

\begin{proof}
	~~First notice that the system (\ref{system}) is not well defined  when $\nabla f=0$. The trajectory $x(t; \zeta,x_0)$ will terminate at these points
	upon finding them. Therefore, the maximal interval $T_{\zeta,x_0}$ may be finite. To overcome this difficulty, we use an equivalent Y$-$system:
	
	\begin{equation}
		\left\{
		\begin{split}
			&\frac{dy}{d\tau}=|\nabla f| \textbf{s}_\zeta(x),\\
			& y|_{\tau=0} = x_0.
		\end{split}
		\right.
		\label{Y-system}
	\end{equation}
	The system (\ref{Y-system}) has the same orbit as (\ref{system}) but different parameterization. It has the solution $y(\tau; \zeta,x_0)$ with infinite existence interval  in the whole $\mathbb{R}^n$ by Picard's existence theorem.
	
	For $ \zeta \in [0,1) $, consider an integral along $y(\tau; \zeta, x_0)$, \begin{equation*}
		f(y(T; \zeta,x_0))= f(x_0)-\int_0^{T}|\nabla f|^2(1 +\zeta\cos \theta)(y(\tau; \zeta,x_0))d\tau.
	\end{equation*}
	Lemma \ref{lemma1} ensures
	\begin{equation*}
		\int_0^{T}|\nabla f|^2(1 +\zeta\cos \theta)(y(\tau; \zeta,x_0))d\tau= f(x_0)-f(y(T; \zeta,x_0))\leq M \end{equation*} with some positive number $M$ independent of $T$. Hence
	
	\begin{equation*}
		\int_0^{\infty}|\nabla f|^2(1 +\zeta\cos \theta)(y(\tau; \zeta,x_0))d\tau\leq M.
	\end{equation*}
	\textcolor[rgb]{0,0,0} {With $1+\zeta\cos \theta\geq 1-\zeta$, a similar method to the proof for Theorem \ref{theorem1}(i) shows that}
	\begin{equation*}
		\lim_{\tau\rightarrow +\infty}|\nabla f|^2=0.
	\end{equation*}
	This proves that $y(\tau, \zeta, x_0)$ approaches the connected subset of the critical points. An isolated condition  makes sure that
	\begin{equation*}
		\lim_{\tau\rightarrow +\infty}y(\tau; \zeta, x_0)=x_c,
	\end{equation*}
	for some critical point $x_c$ of the objective function.
	Notice that along the trajectory,
	\begin{equation*}
		t=\int_0^{\tau}|\nabla f|d \tau,
	\end{equation*}
	hence
	\begin{equation*}
		\lim_{t\rightarrow T_{\zeta,x_0}^-}x(t; \zeta,x_0)=x_c,~~~~~~~~~~ \forall \zeta \in [0,1).
	\end{equation*}
\end{proof}

\begin{myremark}
	The global behavior of the system (\ref{system}) looks like a gradient flow. Obviously, as $\zeta = 0$, the system reduces to the normalized and negative gradient flow of the objective function. Especially, the present trajectory is homotopic with the gradient descent trajectory for $\zeta \in (0,1)$ by Theorem \ref{theorem2} and the continuous dependence of the solutions to differential equations on parameters.
	\label{remark1}
\end{myremark}

\begin{myremark}
	The analytical trajectory may find a critical point of constraint $g$ without condition (A3). In practical implementations, we suggest using the gradient descent $-\nabla f$ to escape these critical points.
\end{myremark}
\vskip 2mm

\begin{mylemma}
	Under assumptions (A1)$-$(A4), the trajectory $x(t; \zeta=1, x_0)$ converges, as $t\rightarrow+\infty$, into any $\mu-$neighborhood $\Theta_\mu$ of the central path $L$ with $\mu \in [0,1)$, given that the initial design $x_0 \in \Omega$.
	\label{lemma2}
\end{mylemma}

\begin{proof}
	~~Along the trajectory $x(t; 1,x_0)$, we derive the constraint function $g$ with respect to $t$ and have
	\begin{equation}\label{zeta1}
		\frac{d }{dt}g(x(t; 1,x_0))= -|\nabla g|(1 +\cos \theta)\leq 0.
	\end{equation}
	By Picard's existence theorem, the vector field $\textbf{s}_1$ keeps $C^1$ continuity in a neighborhood of the trajectory $x(t; 1,x_0)$ and $T_{1,x_0}=+\infty$.
	Integrating (\ref{zeta1}) gives
	\begin{equation}
		\int_0^{+\infty}|\nabla g|(1 +\cos \theta)dt = g(x(T_{1, x_0}; 1, x_0)) - g(x(0; 1, x_0)) <+\infty.
	\end{equation}
	Based on Lemma \ref{lemma1} and the assumption (A4), we know that the integral function $|\nabla g|(1 +\cos \theta)$ is continuously differentiable with respect to $t$. Same procedure as in the proof for Theorem \ref{theorem1} shows that
	\begin{equation}
		\lim_{t\rightarrow +\infty}|\nabla g|(1 +\cos \theta)=0,
		\label{eq:continuous_argument}
	\end{equation}
	resulting with (A3) that $(1 + \cos \theta) \rightarrow  0$ as $t \rightarrow + \infty$. In other word, trajectory $x(t; 1,x_0)$ comes into any $\mu-$neighborhood $\Theta_\mu$ of the central path $L$, $\mu \in [0,1)$.
\end{proof}

\begin{mytheorem}
	Under assumptions (A1)$-$(A4),
	the trajectory $x(t;1,x_0)$ converges to a point $\widehat{x}$ on central path $L$ generally, i.e.,
	
	\begin{equation}
		\lim_{t\rightarrow +\infty}x(t; 1,x_0)=\widehat{x}\in L,
	\end{equation}
	provided $x_0 \in \Omega$.
	\label{theorem3}
\end{mytheorem}

\begin{proof}
	~~According to (\ref{eq:continuous_argument}) we have
	$$
	\lim_{t\rightarrow +\infty}\cos\theta=-1.
	$$
	Choose a sequence $t_j\rightarrow +\infty$ as $j\rightarrow +\infty$ with
	\begin{equation}
		\lim_{j\rightarrow +\infty}x(t_j; 1,x_0)=\widehat{x},
		\label{sequence}
	\end{equation}
	for some point $\widehat{x}$. Hence
	$$
	\cos\theta|_{\widehat{x}}=-1,
	$$
	and
	\begin{equation}
		\left\{
		\begin{split}
			& f(\widehat{x})= f(x_0)-\int_0^{+\infty}|\nabla f|(1 +\cos \theta)(x(t; 1,x_0))dt, \\
			& g(\widehat{x})= g(x_0)-\int_0^{+\infty}|\nabla g|(1 +\cos \theta)(x(t; 1,x_0))dt.\\
		\end{split}
		\right.
		\label{eq:property_central_point}
	\end{equation}
	This says that the point $\widehat{x}\in L$, and the values $f(\widehat{x}), g(\widehat{x})$ depend only on $x_0$. In general, the point with these conditions is unique\footnote{If the central path and the contour of the constraint function $g$ intersect transversally, then the point $\widehat{x}$ is unique. By Sard's theorem,  transversal intersection occurs with the probability of 1.}. Therefore
	\begin{equation}
		\lim_{t\rightarrow +\infty}x(t; 1,x_0)=\widehat{x}\in L.
		\label{eq:converge_to_central_point}
	\end{equation}
\end{proof}

\begin{mytheorem}
	Under assumptions (A1)$-$(A4), the trajectory  $x(t; \zeta,x_0)$ is convergent to the trajectory  $x(t; 1, x_0)$ uniformly for $t\in (0, \alpha\log \frac{1}{1-\zeta})$ for some $\alpha>0$.
	\label{theorem4}
\end{mytheorem}

\begin{proof}
	~~By system (\ref{system}), we have the integral equation:
	\begin{equation*}
		x(T;\zeta,x_0)=x_0+\int_0^T \textbf{s}_\zeta(x(t;\zeta,x_0)dt.
	\end{equation*}
	So
	\begin{equation}
		\begin{split}
			& x(T;1,x_0)-x(T;\zeta,x_0)= \int_0^T \left[\textbf{s}_1(x(t;1,x_0)- \textbf{s}_\zeta(x(t;\zeta,x_0)\right]dt \\
			&\hskip 15mm = -\int_0^T \left[\frac{\nabla f}{|\nabla f|}(x(t;1,x_0))-\frac{\nabla f}{|\nabla f|}(x(t;\zeta,x_0))\right]dt\\
			&\hskip 19mm-\zeta\int_0^T \left[\frac{\nabla g}{|\nabla g|}(x(t;1,x_0))-\frac{\nabla g}{|\nabla g|}(x(t;\zeta,x_0))\right]dt\\
			&\hskip 19mm-(1-\zeta)\int_0^T \left[\frac{\nabla g}{|\nabla g|}(x(t;1,x_0))\right]dt.
		\end{split}
	\end{equation}
	In the bounded domain $\{x\in \mathbb{R}^n: f(x)\leq f(x_0), g(x) \leq 0\}$, there is a constant $M$ such that
	\begin{equation}
		\begin{split}
			&\left|\frac{\nabla f}{|\nabla f|}(x(t;1,x_0))- \frac{\nabla f}{|\nabla f|}(x(t;\zeta,x_0))\right|\leq M \left|x(t;1,x_0)- x(t;\zeta,x_0)\right|, \\
			&\left|\frac{\nabla g}{|\nabla g|}(x(t;1,x_0))- \frac{\nabla g}{|\nabla g|}(x(t;\zeta,x_0))\right|\leq M \left|x(t;1,x_0)- x(t;\zeta,x_0)\right|.
		\end{split}
	\end{equation}
	Set
	$$\psi_\zeta(t)=\left|x(t;1,x_0)- x(t;\zeta,x_0)\right|,~~ \Psi_\zeta(T)=\int_0^T\psi_\zeta(t)dt.$$
	Then
	\begin{equation}
		\psi_\zeta(T)\leq 2M\int_0^T\psi_\zeta(t)dt+(1-\zeta)T.
	\end{equation}
	Hence
	\begin{equation}
		\Psi_\zeta^\prime(t)\leq 2M\Psi_\zeta(t)+(1-\zeta)t.
	\end{equation}
	Or
	\begin{equation}
		\left(e^{-2Mt}\Psi_\zeta(t)\right)^\prime\leq (1-\zeta)te^{-2Mt}.
	\end{equation}
	Integrating this inequality gives
	\begin{equation}
		\Psi_\zeta(T)\leq (1-\zeta)e^{2MT}\int_0^T te^{-2Mt}dt.
	\end{equation}
	It follows that for $T\leq -\alpha\log (1-\zeta),$ we have
	\begin{equation}
		\begin{split}
			&\psi_\zeta(T)\leq 2M(1-\zeta)\int_0^T te^{2M(T-t)}dt+(1-\zeta)T\\
			&\hskip 10mm \leq (1-\zeta)Te^{2MT}\leq -\alpha(1-\zeta)^{1-2M\alpha}\log (1-\zeta).
		\end{split}
	\end{equation}
	Therefore, for $t\in (0,-\alpha \log (1-\zeta))$
	\begin{equation}
		\left|x(t;1,x_0)- x(t;\zeta,x_0)\right|\leq -\alpha(1-\zeta)^{1-2M\alpha}\log (1-\zeta).
	\end{equation}
	A choice of $\alpha>0$ with $1-2M\alpha>0$ completes the proof of the theorem.
\end{proof}

\begin{myremark}
	Theorem \ref{theorem4} indicates that the time in which the trajectory needs to arrive at the boundary of the feasible set must be larger than $\alpha \log \frac{1}{1-\zeta}$.
\end{myremark}

\begin{mytheorem}
	Let $x_\zeta^\sharp$ be the first point where the trajectory reaches the boundary of the feasible set, then
	\begin{equation}
		x_\zeta^\sharp\in \{x: g(x) = 0, \cos\theta\leq -\zeta \}.
		\label{eq:theorem_5}
	\end{equation}
	This is to say that the point $x_\zeta^\sharp$ belongs to the closure of the $\zeta$-neighborhood of the central path. Especially, the limit of $x_\zeta^\sharp$ as $\zeta\rightarrow 1^-$  is at the intersection of the central path and the boundary of the feasible set.
	\label{theorem5}
\end{mytheorem}

\begin{proof}
	~~Let $x_\zeta^\sharp=x(t^\sharp;\zeta,x_0)$, by the choice of the point $x_\zeta^\sharp$,
	\begin{equation*}
		\left\{
		\begin{split}
			&g(x(t^\sharp;\zeta,x_0))=0,\\
			& g(x(t;\zeta,x_0))<0,~~ t<t^\sharp.
		\end{split}\right.
	\end{equation*}
	Hence
	$$
	\frac{d }{dt}g(x(t^\sharp; \zeta,x_0))\geq 0.
	$$
	By (\ref{deform}),
	$$ |\nabla g|(\zeta +\cos \theta)\leq 0. $$
	Therefore,
	$$ \cos \theta\leq -\zeta. $$
	This ends the proof.
\end{proof}

\begin{myremark}
	According to Theorem \ref{theorem5}, obviously, we have $\cos \theta \rightarrow -1^+$ as $\zeta \rightarrow 1^-$. The normalized centrality condition (\ref{eq:normalized_central_path_condition_single_constraint})
	$$\frac{\nabla f}{|\nabla f|} + \frac{\nabla g}{|\nabla g|}  = 0 $$
	can be satisfied. Given that $x_\zeta^\sharp$ is a point on the boundary of the feasible set, the first-order necessary conditions (KKT conditions) are therefore satisfied. This is straightforward as the Lagrange multiplier $\lambda^\star$ associated with the Lagrangian function $\mathcal{L}(x,\lambda)$ for the considered problem is 
	$$ \lambda^\star = \frac{|\nabla f|}{|\nabla g|},$$
	and 
	$$\textcolor[rgb]{0,0,0}{\nabla_x \mathcal{L}}(x_\zeta^\sharp, \lambda^\star) = \nabla f + \lambda^\star \nabla g = 0.$$
	Under assumption (A2) and (A3), we have $\lambda^\star > 0$. With $g(x_\zeta^\sharp) = 0$ strict complementarity holds. To find out whether the point $x_\zeta^\sharp$ is a local solution, one needs to check the second-order sufficient conditions. We discuss this in the next section.
	\label{remark_fo}
\end{myremark}

\textcolor[rgb]{0,0,0}{
	\begin{myremark}
		Based on Theorem \ref{theorem5}, we propose an error measure $\epsilon > 0$ for the optimization solutions
		\begin{equation}
			\epsilon = 1-\zeta.
		\end{equation}
		According (\ref{eq:theorem_5}), we then have $$x_\zeta^\sharp\in \{x: g(x) = 0, \cos\theta\leq -1 + \epsilon \}.$$
		The KKT conditions are satisfied for $x_\zeta^\sharp$ when $\cos \theta = -1$ as mentioned in Remark \ref{remark_fo}. A small value $\epsilon>0$ seems to be a natural error measure for the present method. Additionally, $\epsilon$ is defined using the intrinsic parameter $\zeta$ that determines the shape of the optimization trajectory. With the error measure $\epsilon$ we can interpret Theorem \ref{theorem1}(ii) and Theorem \ref{theorem4} as follows. Theorem \ref{theorem1}(ii) implies the time for the trajectory to find a first-order $\epsilon$-optimal solution is $\mathcal{O}(1/\epsilon)$, which is the ergodic rate of convergence for first-order methods; Theorem \ref{theorem4} implies the trajectory is convergent to the trajectory $x(t;1, x_0)$ uniformly for $t\in (0, \alpha \log(1/\epsilon))$.
		\label{remark_error}
	\end{myremark}	
}


\section{Local convergence analysis}
\label{sec:local_analysis}

Consider optimization problem (\ref{eq:Optimization_Problem}) with a single inequality constraint. For simplicity, let the origin be a point on the central path, and $x_n$ lies in the direction of $\frac{\nabla f}{|\nabla f|}$, i.e.,
\begin{equation}
	\nabla f(0,...,0) = \left[ 0,...,0,\frac{\partial f}{\partial x_n}\right] \ne \mathbf{0}.
\end{equation}
By the implicit function theorem, we have a function $x_n = \phi(x_1,...,x_{n-1})$, which satisfies
\begin{equation}
	f(x_1, \cdots, x_{n-1}, \phi(x_1, \cdots, x_{n-1}))\equiv f(0, \cdots, 0, 0).
\end{equation}
Furthermore:
\begin{itemize}
	\item[\textcircled{1}] At the origin,
	\begin{equation}
		\phi(0,...,0) = 0;
		\label{eq:phi_condition_1}
	\end{equation}
	\item[\textcircled{2}] In a neighborhood of the origin,
	\begin{equation}
		\frac{\partial \phi}{\partial x_k} = - \frac{\partial f}{\partial x_k} / \frac{\partial f}{\partial x_n}.
		\label{eq:phi_condition_2}
	\end{equation}
	Thus we have
	\begin{equation}
		\frac{\partial \phi(\mathbf{0})}{\partial x_k} = 0, ~~ 1 \leq k \leq n-1.
	\end{equation}
	
\end{itemize}

\textcolor[rgb]{0,0,0}{The contours of the objective function is the graph of the implicit function
	$x_n=\phi(x_1,...,x_{n-1})$.  For $n=2$, the curvature formula gives
	$$\kappa_f=\left.\frac{\phi^{\prime\prime}(x_1)}
	{(1+(\phi^\prime(x_1))^2)^\frac{3}{2}}\right|_{x_1=0}=\phi^{\prime\prime}(0).
	$$
}

\textcolor[rgb]{0,0,0}{For $n>2$, we choose a direction $\omega=(\omega_1,\cdots, \omega_{n-1})$ with $|\omega|=1$.
	\textcolor[rgb]{0,0,0}{Then, in the two dimensional subspace defined by $\mathbb{R}^2=\{\omega t,x_n\}= \{\omega_1 t,\cdots, \omega_{n-1} t, x_n\}$,} the function $x_n=\phi(\omega t)$ has curvature in direction $\omega$,
	\begin{equation}\kappa_f^\omega=\left.\frac{d^2\phi(\omega t)}{dt^2}\right|_{t=0}
		=\sum\limits_{k=1,j=1}^{n-1} \frac{\partial^2 \phi(\mathbf{0})}{\partial x_k \partial x_j} \omega_k \omega_j.
		\label{directioncurvaturephi}
	\end{equation}
}

Similar analysis on $\psi$, which is deduced from the implicit function $g=\text{constant}$, gives:
\begin{equation}
	\kappa_g^\omega=\left.\frac{d^2\psi(\omega t)}{dt^2}\right|_{t=0}
	=\sum\limits_{k=1,j=1}^{n-1} \frac{\partial^2 \psi(\mathbf{0})}{\partial x_k \partial x_j} \omega_k \omega_j.
	\label{directioncurvaturepsi}
\end{equation}\vskip 2mm
Now the conjectured \textit{relative convex condition} (\ref{eq: relative_convex}) in the two dimensional subspace $\mathbb{R}^2=\{\omega t,x_n\}$ would be $\kappa_f^\omega < \kappa_g^\omega$, i.e.,
\begin{equation}
	\sum\limits_{k=1,j=1}^{n-1} \frac{\partial^2 \phi(\mathbf{0})}{\partial x_k \partial x_j} \omega_k \omega_j <
	\sum\limits_{k=1,j=1}^{n-1} \frac{\partial^2 \psi(\mathbf{0})}{\partial x_k \partial x_j} \omega_k \omega_j
	\label{Hessiancomp}
\end{equation}
for any direction $\omega$.
Let $H_{\phi} (\mathbf{0})$ and $H_{\psi} (\mathbf{0})$ be the Hessian matrix of $\phi$ and $\psi$ about variable $\tilde{x}$, respectively. Then, in the sense of positive definite matrix, we have
\begin{equation}\label{semi-positive definition}
	H_{\phi} (\mathbf{0}) \prec H_{\psi} (\mathbf{0}).
\end{equation}

At $(x_1,...,x_{n-1}) = (0,...,0)$, we have
\begin{equation}
	\begin{split}
		\frac{\partial^2 \phi}{\partial x_k \partial x_j} &= \frac{\partial}{\partial x_k} \left(\frac{\partial \phi}{\partial x_j}\right)
		= -\frac{\partial}{\partial x_k} \left(\frac{\partial f}{\partial x_j} / \frac{\partial f}{\partial x_n}\right) \\
		&= - \left( \frac{\partial f}{\partial x_n} \right)^{-1} \frac{\partial^2 f}{\partial x_k \partial x_j}.
	\end{split}
\end{equation}
Notice that
\begin{equation}
	\nabla^2_{\tilde{x}} f = \left( \frac{\partial^2 f}{\partial x_k \partial x_j}  \right)_{ 1 \leq k,j \leq n-1}.
\end{equation}
Thus, we have
\begin{equation}
	H_{\phi}(\mathbf{0}) = - \left( \frac{\partial f}{\partial x_n} \right)^{-1} \nabla^2_{\tilde{x}} f(\mathbf{0}).
\end{equation}
Similarly, we have for $H_{\psi} (\mathbf{0})$
\begin{equation}
	H_{\psi} (\mathbf{0}) = - \left( \frac{\partial g}{\partial x_n} \right)^{-1} \nabla^2_{\tilde{x}} g(\mathbf{0}).
\end{equation}
Due to the positive definiteness (\ref{semi-positive definition}), we have
\begin{equation}
	- \left( \frac{\partial f}{\partial x_n} \right)^{-1} \nabla^2_{\tilde{x}} f(\mathbf{0}) \prec - \left( \frac{\partial g}{\partial x_n} \right)^{-1} \nabla^2_{\tilde{x}} g(\mathbf{0}).
\end{equation}
Recall at the origin we have
$$ \frac{\partial f}{\partial x_n} = | \nabla f|, ~~\frac{\partial g}{\partial x_n} = -| \nabla g|.$$
Therefore we have
\begin{equation}
	\frac{1}{|\nabla f|} \nabla^2_{\tilde{x}} f(\mathbf{0}) + \frac{1}{|\nabla g|} \nabla^2_{\tilde{x}} g(\mathbf{0}) \succ 0.
	\label{eq:relative_semi_convex}
\end{equation}
Set
\begin{equation}
	\tilde{C} = \frac{1}{|\nabla f|} \nabla^2_{\tilde{x}} f(\mathbf{0}) + \frac{1}{|\nabla g|} \nabla^2_{\tilde{x}} g(\mathbf{0}).
	\label{eq:relative_convex_matrix}
\end{equation}

\begin{mydefinition}
	A point $x$ on the central path is called nondegenerate if $\tilde{C}(x)$ is invertible; it is \textit{relative convex} if $\tilde{C}(x)$ is a positive definite matrix.
\end{mydefinition}
\vskip 2mm

\begin{myremark}
	Let a point $x^\star$ satisfy the KKT conditions. The relative convex condition $\tilde{C}(x^\star) \succ 0 $ is equivalent to the second-order sufficient conditions for constrained optimization. This is straightforward as the matrix 
	$$ \tilde{H} =  |\nabla f(x^\star) | \tilde{C}(x^\star) = \nabla^2_{\tilde{x}} f(\mathbf{0}) + \frac{|\nabla f|}{|\nabla g|} \nabla^2_{\tilde{x}} g(\mathbf{0}) \succ 0$$
	is equivalent to the \textit{projected Hessian} being positive definite \cite[p.~348]{Nocedal}, with a Lagrange multiplier $\lambda^\star$ satisfying the KKT conditions and strictly complementarity holding,
	$$ \lambda^\star = \frac{|\nabla f|}{|\nabla g|} > 0, ~g(x^\star) = 0.$$
	Here, we deduce the relative convex condition from the perspective of the difference in the curvatures of the function contours in the feasible set. It is defined on the central path and may be seen as a perturbed version of the second-order sufficient conditions.
	\label{remark_so}
\end{myremark}
\vskip 2mm
%


At the origin, the Jacobian matrix for $\frac{\nabla f}{|\nabla f|}$ reads
\begin{equation}
	J_f^n = \left( \frac{\partial}{\partial x_j}  \left( \frac{\partial_{x_i} f}{|\nabla f|} \right)\right)_{1\leq i,j\leq n}.
\end{equation}
For $ 1 \leq i \leq n-1$, $\partial_{x_i} f = 0$ at the origin, thus
\begin{equation}
	\left( \frac{\partial}{\partial x_j}  \left( \frac{\partial_{x_i} f}{|\nabla f|} \right)\right) = \frac{\partial_{x_i} \partial_{x_j} f}{|\nabla f|}.
\end{equation}
Recall also at the origin we have
\begin{equation}\left\{
	\begin{split}
		\frac{\nabla f}{|\nabla f|} &= (\frac{\partial_{x_1} f}{|\nabla f|},...,\frac{\partial_{x_{n-1}} f}{|\nabla f|},\frac{\partial_{x_n} f}{|\nabla f|}) = (0,...,0,1), \\
		\frac{\nabla g}{|\nabla g|} &= (\frac{\partial_{x_1} g}{|\nabla g|},...,\frac{\partial_{x_{n-1}} g}{|\nabla g|},\frac{\partial_{x_n} g}{|\nabla g|}) = (0,...,0,-1).
	\end{split}
	\right.
\end{equation}
At the origin, $\frac{\partial_{x_n} f}{|\nabla f|} $ has a maximum value $1$, resulting in its derivatives being zero:
\begin{equation}
	\frac{\partial}{\partial x_k} \left( \frac{\partial_{x_n} f}{|\nabla f|}\right) = 0, ~ 1\leq k \leq n.
\end{equation}
Thus, the Jacobian matrix $J_f^n $ for $\frac{\nabla f}{|\nabla f|}$ at the origin reads
\begin{equation}
	J_f^n =
	\begin{bmatrix}
		\begin{array}{ccccc|c}
			& & & & & \frac{\partial}{\partial x_n} \left( \frac{\partial_{x_1} f}{|\nabla f|}\right) \\
			& & J_f^{n-1} & & &  \vdots \\
			& & & & & \frac{\partial}{\partial x_n} \left( \frac{\partial_{x_{n-1}} f}{|\nabla f|}\right) \\
			\hline
			0 & & \cdots & & 0 & 0
		\end{array}
	\end{bmatrix}
\end{equation}
where $ J_f^{n-1}$ is an $(n-1) \times (n-1)$ matrix:
\begin{equation}
	J_f^{n-1} = \frac{1}{|\nabla f|} \left( \partial_{x_i} \partial_{x_j} f\right)_{ 1\leq i,j \leq n-1}.
\end{equation}

Similarly, we have the Jacobian matrix $J_g^n$ for $\frac{\nabla g}{|\nabla g|}$ at the origin that reads
\begin{equation}
	J_g^n=
	\begin{bmatrix}
		\begin{array}{ccccc|c}
			& & & & & \frac{\partial}{\partial x_n} \left( \frac{\partial_{x_1} g}{|\nabla g|}\right) \\
			& & J_g^{n-1} & & &  \vdots \\
			& & & & & \frac{\partial}{\partial x_n} \left( \frac{\partial_{x_{n-1}} g}{|\nabla g|}\right) \\
			\hline
			0 & & \cdots & & 0 & 0
		\end{array}
	\end{bmatrix}
\end{equation}
where $ J_g^{n-1}$ is also an $(n-1) \times (n-1)$  matrix:
\begin{equation}
	J_g^{n-1} = \frac{1}{|\nabla g|} \left( \partial_{x_i} \partial_{x_j} g\right)_{1\leq i,j \leq n-1}.
\end{equation}
So we get matrix $\tilde{C}=J_f^{n-1} + J_g^{n-1}$ once more again. 

By choosing a new coordinate system, say  $\tilde{x}= (x_1, ...,x_{n-1})$ again, whose basis vectors are the eigenvectors $\mathbf{v}$ of $\tilde{C}$, the matrix $\tilde{C}$ transforms to a diagonal matrix $\tilde{C}_\lambda$ in the new coordinate system:
\begin{equation}
	\tilde{C}_\lambda =
	\begin{bmatrix}
		\lambda_1 & & 0\\
		& \ddots & \\
		0 & & \lambda_{n-1}\\
	\end{bmatrix}
\end{equation}
with $\lambda_1, ..., \lambda_{n-1}$, which are the eigenvalues of the symmetric matrix $\tilde{C}$. Hence
\begin{equation}
	J_f^n + J_g^n =
	\begin{bmatrix}
		\begin{array}{ccccc|c}
			\lambda_1 & & & & & \mu_1 \\
			& & \ddots & & &  \vdots \\
			& & & & \lambda_{n-1} & \mu_{n-1} \\
			\hline
			0 & & \cdots & & 0 & 0
		\end{array}
	\end{bmatrix}
\end{equation}
where $\mu_i = \frac{\partial}{\partial x_n} \left( \frac{\partial x_i f}{|\nabla f|}\right) + \frac{\partial}{\partial x_n} \left( \frac{\partial x_i g}{|\nabla g|}\right)$.

Assume that at the origin, the tangent of the central path $\mathcal{L}: \frac{\nabla f}{|\nabla f|} + \frac{\nabla g}{|\nabla g|} = 0$ writes as
\begin{equation}
	x_i = l_i x_n, ~ 1 \leq i \leq n-1.
	\label{eq:cp_tangent}
\end{equation}
The directional derivative of $\frac{\nabla f}{|\nabla f|} + \frac{\nabla g}{|\nabla g|} $ along the tangent of cental path $L$ is zero, we have
\begin{equation}
	\lambda_i l_i + \mu_i = 0, ~ 1 \leq i \leq n-1.
\end{equation}
We denote at origin:
\begin{equation}\left\{
	\begin{split}&a_{ij} = \frac{\partial_{x_i} \partial_{x_j} g}{|\nabla g|}, ~1 \leq i,j \leq n-1,\\
		&b_i = \partial_{x_n} \left( \frac{\partial_{x_i} g}{|\nabla g|}\right).
	\end{split}
	\right.
\end{equation}
The Taylor formula of the negative search direction $\frac{\nabla f}{|\nabla f|} + \zeta \frac{\nabla g}{|\nabla g|}$ reads
\begin{equation}
	\begin{split}
		\textcolor[rgb]{0,0,0}{\left(
			\frac{\nabla f}{|\nabla f|} + \zeta \frac{\nabla g}{|\nabla g|}\right)^T} =
		&\begin{bmatrix}
			\begin{array}{ccccc|c}
				\lambda_1 & & & & & -\lambda_1 l_1 \\
				& & \ddots & & &  \vdots \\
				& & & & \lambda_{n-1} & -\lambda_{n-1} l_{n-1} \\
				\hline
				0 & & \cdots & & 0 & 0
			\end{array}
		\end{bmatrix}
		\begin{bmatrix}
			x_1\\
			\vdots\\
			x_{n-1}\\
			x_n
		\end{bmatrix}\\
		+ (\zeta -1) &\begin{bmatrix}
			\begin{array}{ccc|c}
				& & &  b_1 \\
				& (a_{ij})_{(n-1)\times(n-1)} & &  \vdots \\
				& & &  b_{n-1} \\
				\hline
				0 & \cdots & 0 & 0
			\end{array}
		\end{bmatrix}
		\begin{bmatrix}
			x_1\\
			\vdots\\
			x_{n-1}\\
			x_n
		\end{bmatrix}\\
		+ &\begin{bmatrix}
			0\\
			\vdots\\
			0\\
			1-\zeta
		\end{bmatrix} + o(\rho).
	\end{split}
	\label{Taylor}
\end{equation}

Summarizing the above analysis, we have
\vskip 2mm
\begin{mylemma}
	If a point on the central path $L$ is nondegenerate, then the system (\ref{system})
	may locally be considered as a perturbation of the linear equations:
	\begin{equation}
		\left\{
		\begin{split}
			&\frac{dx_i}{dt} = -\lambda_i x_i + \lambda_i l_i x_n + (1- \zeta) \left(\sum\limits_{j = 1}^{n-1}a_{ij} x_j + b_i x_n\right),~ 1 \leq i \leq n-1 \\
			&\frac{dx_n}{dt} = -(1 - \zeta).
		\end{split}
		\right.
		\label{linearized}
	\end{equation}
\end{mylemma}

In the following, we use the matrix representations
\begin{equation}\left\{
	\begin{split}
		&\Lambda_\zeta = \left( \lambda_i \delta_{ij} -(1-\zeta) a_{ij} \right)_{(n-1) \times (n-1)}, \\
		&B_\zeta = \left(\lambda_i l_i + (1-\zeta) b_i \right)_{(n-1)\times 1}, \\
	\end{split}\right.
\end{equation}
where $\delta_{ij}$ is the Kronecker-delta.

\vskip 2mm
\begin{mylemma}
	Given an initial point $(x_1^0, x_2^0,\cdots,x_n^0), ~x_n^0 >0$, the linearized equation system (\ref{linearized}) has the solution
	\begin{equation}
		\left\{
		\begin{split}
			&\tilde{x} = C_0 e^{-\Lambda_\zeta t} + \Lambda_\zeta^{-1} B_\zeta x_n(t) + (1-\zeta) \Lambda_\zeta^{-2} B_\zeta, \\
			&x_n(t) = x_n^0 - (1-\zeta)t
		\end{split}
		\right.
		\label{limitation}
	\end{equation}
	with
	$$C_0=\tilde{x}_0-\Lambda_\zeta^{-1} B_\zeta x_n^0 -(1-\zeta) \Lambda_\zeta^{-2} B_\zeta.$$
\end{mylemma}

\vskip 2mm
\begin{proof}
	~~Given an initial point $(x_1^0, x_2^0,\cdots,x_n^0), ~x_n^0 >0$, we have
	$$
	x_n(t) = x_n^0 - (1-\zeta)t.
	$$
	Therefore, for $1\leq i \leq n-1$ we have
	\begin{equation}
		\frac{dx_i}{dt} = -\lambda_i x_i + (1-\zeta) \sum\limits_{j=1}^{n-1} a_{ij} x_j + [\lambda_i l_i + (1-\zeta) b_i] [x_n^0 -(1-\zeta)t].
		\label{eq:ODE_index_form}
	\end{equation}
	
	We write the ODE (\ref{eq:ODE_index_form}) in matrix form,
	\begin{equation}
		\frac{d\tilde{x}}{dt} = - \Lambda_\zeta \tilde{x} + (x_n^0 -(1-\zeta) t) B_\zeta.
		\label{eq:ode_matrix}
	\end{equation}
	(\ref{eq:ode_matrix}) has the solution
	\begin{equation}
		\tilde{x} = C_0 e^{-\Lambda_\zeta t} + \Lambda_\zeta^{-1} B_\zeta x_n(t) + (1-\zeta) \Lambda_\zeta^{-2} B_\zeta.
		\label{eq:ode_solution_hd}
	\end{equation}
	where $C_0$ is a vector of constants depending on the initial point $(x_1^0,\cdots,x_n^0)$.
\end{proof}

\begin{mytheorem}
	If $\tilde{C}_\lambda$ is positive definite, then there is a positive $\zeta_0<1$ such that  the solution of the system (\ref{linearized}) has the asymptotic line
	\begin{equation}
		\tilde{x} = \Lambda_\zeta^{-1} B_\zeta x_n + (1-\zeta) \Lambda_\zeta^{-2} B_\zeta,
		\label{infty}
	\end{equation}
	as $t\rightarrow +\infty$
	for any $\zeta_0<\zeta<1$.
\end{mytheorem}

\vskip 2mm

\begin{proof}
	~~The distance of a point on the trajectory (\ref{eq:ode_solution_hd}) to the asymptotic line (\ref{infty}) is at most $|C_0 e^{-\Lambda_\zeta t} |$. Choose a $\zeta_0$ such that $\Lambda_\zeta$ is a positive definite matrix, then the distance goes to zero as
	$t\rightarrow +\infty$ for any $\zeta_0<\zeta<1$. This is a proof of the theorem.
\end{proof}

To observe the behavior of the linearized equation (\ref{linearized}) when $\zeta\rightarrow 1^-$, we eliminate the variable $t$ with $t = \frac{x_n^0-x_n}{1-\zeta}$, and rewrite the solution (\ref{eq:ode_solution_hd}) as
\begin{equation}
	\tilde{x} = C_0 e^{-\Lambda_\zeta \frac{(x_n^0-x_n)}{(1-\zeta)}} + \Lambda_\zeta^{-1} B_\zeta x_n + (1-\zeta) \Lambda_\zeta^{-2} B_\zeta.
	\label{non-t}
\end{equation}
Notice that $x_n^\prime(t)<0$, and we have $x_n(t) = x_n^0 - (1-\zeta)t<x_n^0.$

\vskip 2mm	
\begin{mytheorem}
	Let $\zeta\rightarrow 1^-$  and the point on the central path be relative convex, then the solution of system (\ref{linearized}) converges to the tangent of the central path $\mathcal{L}$ defined in (\ref{eq:cp_tangent}).
\end{mytheorem}	
\vskip 2mm

\begin{proof}
	~~Following the observation described above, $x_n<x_n^0.$  Let $\zeta \rightarrow 1^-$, then $ (x_n^0-x_n)/(1-\zeta) \rightarrow +\infty$. With the relative convex condition, we have $\lambda_i > 0$. If $ \zeta \rightarrow 1^-$, the eigenvalues of $\Lambda_\zeta$ $> \frac{1}{2} \min (\lambda_i) > 0$. Resulting in $\Lambda_1 \succ 0$ by denoting $\Lambda_{\zeta = 1} = \Lambda_1$.
	Therefore,    
	\begin{equation*}
		e^{-\Lambda_\zeta \frac{(x_n^0-x_n)}{(1-\zeta)}} \rightarrow 0,~~~  \mbox{if}~ \zeta \rightarrow 1^-.
	\end{equation*}
	The solution (\ref{non-t}) converges to
	\begin{equation}
		\tilde{x}= \Lambda_1^{-1} B_1 x_n + (1-1) \Lambda_1^{-2} B_1 = \left( l_i x_n \right)_{(n-1) \times 1},
	\end{equation}
	which is the tangent of the central path $\mathcal{L}$ as is given in (\ref{eq:cp_tangent}).
\end{proof}

Similar method gives:

\vskip 2mm	
\begin{mytheorem}
	Suppose a point $x$ on the central path is nondegenerate and $\tilde{C}_\lambda(x)$ has at least one negative eigenvalue. Let $\zeta\rightarrow 1^-$, then the solution of system (\ref{linearized}) will leave a neighborhood of the central path.
	\label{theorem8}
\end{mytheorem}
\vskip 2mm

\begin{myremark}
	The present trajectory converges to a central path with $\zeta \rightarrow 1^-$ under the relative convex condition. In Remark \ref{remark_fo}, we state that the first-order necessary conditions can be obtained as the trajectory reaches the boundary of the feasible set as $\zeta \rightarrow 1^-$. In Remark \ref{remark_so}, we state that the second-order sufficient conditions is equivalent to the relative convex condition at the boundary of the feasible set. Combining both statements we conclude: by choosing $\zeta \rightarrow 1^-$, the present trajectory approaches a local solution by traversing along a central path. 
\end{myremark}



\section{The method for multiple constraints: a formulation based on the logarithmic barrier function}
\label{sec:multiple_constraint}
Recall the logarithmic barrier function for the problem (\ref{eq:Optimization_Problem}),
$$\Phi(x) = - \sum_{i = 1}^{m} \log (-g_i(x)), ~ i = 1,...,m.$$
We first notice that the barrier function grows without bound if $g_i(x) \rightarrow 0^-$ and it is not differentiable at the boundary of the feasible set. To overcome this difficulty, we consider a subset of the feasible set,
\begin{equation}
	\Omega_M=\{x: \Phi(x) \leq M \},
\end{equation}
where M is a sufficiently large positive number so that $\Omega_M$ approximates the original feasible set $\Omega$.
The barrier function is twice continuously differentiable in $\Omega_M$ and thus fulfills the assumption (A4). 

Set $$ G(x) = \Phi(x) - M, ~ x \in \Omega_M. $$
The original problem (\ref{eq:Optimization_Problem}) can be approximately reformulated as
\begin{equation}
	\begin{split}
		&\textnormal{minimize}~~~  f(x), \\
		&\textnormal{subject to}~~ G(x) \leq 0, \\
	\end{split}
	\label{eq:Log_barrier_approximation}
\end{equation}
so that the present system (\ref{eq:modified_search_direction}) may be applied. 

Notice that the global behavior shown in section \ref{sec: global_behavior} is based on the assumption that $\nabla g \neq 0$ in the feasible set $\Omega$. In general, this may not always be the case for the function $G(x), \forall x \in \Omega_M$. To escape the points where $\nabla \Phi(x) = 0$, we suggest using the gradient descent $-\nabla f(x)$. Thus we get the system (\ref{eq: vector_s_multiple}) for solving multiple inequality constrained optimization with $\zeta\in [0,1)$ as
\begin{equation*}
	\frac{d x}{d t}=\left\{
	\begin{split}
		&-\frac{\nabla f}{|\nabla f|}
		-\zeta\frac{\nabla \Phi}{|\nabla \Phi|},  ~~~~~&\text{if}~~\nabla \Phi \neq 0;\\
		& -\nabla f    ~~~~~~~~~~~~~~~~~~~~&\text{if}~~\nabla \Phi = 0.
	\end{split}\right.
\end{equation*}
Note that in computational practice, $\nabla \Phi = 0$ rarely occurs. The second ODE of the above system serves as a safeguard for the present method.
\vskip 2mm

\begin{mycorollary}
	Let $x_\zeta^\sharp$ be the first point where the resulting trajectory of the system (\ref{eq: vector_s_multiple}) reaches the boundary of $\Omega_M$, then
	$$ x_\zeta^\sharp\in \{x:\Phi(x) = M , \cos\theta\leq -\zeta \}.$$
	This is to say that the point $x_\zeta^\sharp$ belongs to the closure of  $\zeta$-neighborhood of the central path. Especially, the limit of $x_\zeta^\sharp$ as $\zeta\rightarrow 1^-$  is at the intersection of the central path and the boundary of the subset $\Omega_M$.
\end{mycorollary}

\begin{proof}
	~~A similar method as used in Theorem \ref{sec: global_behavior} gives the proof.	
\end{proof}

Let $\mathbf{x}^\star \in \Omega_M$ be a point on the central path of the barrier function method, set
\begin{equation}
	\tilde{C}_\Phi = \frac{1}{|\nabla f|} \nabla^2_{\tilde{x}} f(\mathbf{x}^\star) + \frac{1}{|\nabla \Phi |} \nabla^2_{\tilde{x}} \Phi(\mathbf{x}^\star).
	\label{eq:relative_convex_log_barrier}
\end{equation}

\begin{mydefinition}
	A point $\mathbf{x}^\star \in \Omega_M$ on the central path is \textit{relative convex} if $\tilde{C}_\Phi(\mathbf{x}^\star)$ is a positive definite matrix.	
\end{mydefinition}

\vskip 2mm

\begin{mycorollary}
	Let $\zeta\rightarrow 1^-$, and the point $\mathbf{x}^\star$ on the central path be relative convex. Then, the solution of the linearized system of (\ref{eq: vector_s_multiple}) converges to the tangent line of the central path $\mathcal{L}$ at $\mathbf{x}^\star$.
\end{mycorollary}

\begin{proof}
	~~In the subset $\Omega_M$, $\nabla \Phi$ is Lipschitz continuous. Under the assumption (A2), we have $\nabla \Phi \neq 0$ along a central path. A uniformly continuous argument shows that $\nabla \Phi \neq 0$ in a close neighborhood of the central path. A similar method shown in section \ref{sec:local_analysis} gives a proof.
\end{proof}

\textcolor[rgb]{0,0,0}{
	To conclude, using the barrier function formulation for problem (\ref{eq:Optimization_Problem}), the resulting trajectory achieves an approximated local solution that locates on the boundary of the subset $\Omega_M$. Notice, too, that the system (\ref{eq: vector_s_multiple}) does not depend on the choice of $M$, and $\Omega_M$ exhaust $\Omega$, i.e., $\Omega = \cup_{M>0} \Omega_M$. This means that the resulting trajectory keeps approaching the boundary of the original feasible set $\Omega$ by crossing the boundary of any subset $\Omega_M$ dependent on $M$. Therefore, it eases the practical implementation of the method since no extra parameter $M$ needs to be defined for the stopping criterion, but rather, a check for each constraint violation would be sufficient.}

\section{Numerical experiments}
\label{sec:experiments}
When implementing the present method numerically, the canonical first-order optimization procedure is defined by the differential equation
\begin{equation}
	\frac{d x}{d t} = \mathbf{s}_\zeta,
	\label{eq:GD_trajectory}
\end{equation}
where 
\begin{equation*}
	\mathbf{s}_\zeta=\left\{
	\begin{split}
		&-\frac{\nabla f}{|\nabla f|}
		-\zeta\frac{\nabla \Phi}{|\nabla \Phi|},  ~~~~~&\text{if}~~\nabla \Phi \neq 0;\\
		& -\nabla f    ~~~~~~~~~~~~~~~~~~~~&\text{if}~~\nabla \Phi = 0.
	\end{split}\right.
\end{equation*}
In practical applications, this procedure might result in poor performance. We first notice that, according to Theorem \ref{theorem5}, the accuracy of the solutions increases as the parameter $\zeta \rightarrow 1^-$. However, as $\zeta \rightarrow 1^-$, $| \mathbf{s}_\zeta| \downarrow 0$ at the central path, thus resulting in slow convergence rate. 
Another difficulty results from the potential poor scaling of the logarithmic barrier function, which results in ill-conditioned Hessian near the boundary of the feasible set \cite[p.~500-502]{Nocedal}. We will show test examples that suffer from this problem in the following. 

To overcome these difficulties, a step size rule that considers and adapts the parameter $\zeta$ may be needed. 
A good practical self-adaptive $\zeta$ is expected to result in an optimization trajectory that behaves similar to the Mehrotra's practical implementation for IPM \cite{Mehrotra}, which is well-known for its very good performance and its difficulty in deriving a convergence theory. Design and analyze such a self-adaptive $\zeta$ require studies and analysis that exceeds the scope of this manuscript.

Another way to improve the performance might be the implementation of a momentum method or the Nesterov Accelerated Gradient \cite{nesterov1983method} for the present system. Although our first implementations have shown promising performance on some test problems, we leave a systematical study together with the design of a step size rule (that considers self-adaptive $\zeta$) for future work. 

In this work, rather, we propose a more fundamental framework so that the focus is on the numerical behavior of the present trajectory, neglecting any advanced modification. In this framework, we use the fixed parameter $\zeta$. In doing so, we can only get solutions that are $(1-\zeta)$-suboptimal solutions (referred in Theorem \ref{theorem5} and Remark \ref{remark_error}). We let $\mathcal{X}_\zeta$ denote the set of all $(1-\zeta)$-suboptimal solutions,
\begin{equation}
	\mathcal{X}_\zeta = \{x^\sharp : x^\sharp \in \Theta_{\zeta}, \max(g_i(x^\sharp)) = 0 \},
\end{equation}
given that the solutions exist only on the boundary of the feasible set. A different way of defining the suboptimal solutions is given in \cite{skaf2010techniques} and some useful applications using suboptimal solutions are shown.

To get better performance compared to the procedure (\ref{eq:GD_trajectory}), we suggest the following modified dynamical system:
\begin{equation}
	\frac{d x}{d t} = \frac{\mathbf{s}_\zeta}{|\mathbf{s}_\zeta|}.
	\label{eq:NGD_trajectory}
\end{equation}
Referring to the work \cite{murray2019revisiting}, we can say that the systems (\ref{eq:NGD_trajectory}) and (\ref{eq:GD_trajectory}) are topologically equivalent and solutions of (\ref{eq:NGD_trajectory}) are merely arc-length reparameterizations of (\ref{eq:GD_trajectory}) solutions. As the system (\ref{eq:GD_trajectory}) may move slowly at the close neighborhood of a central path, the system (\ref{eq:NGD_trajectory}) moves in the same orbit with constant speed. In \cite{murray2019revisiting}, the authors show that the normalized gradient descent method escapes saddle points ``quickly''. This might be beneficial when solving nonconvex optimization problems, where the gradient of the function $\nabla f(x)$ vanishes at saddle points. In fact, this phenomenon may be similar to the present system (\ref{eq:GD_trajectory}), where $|\mathbf{s}_\zeta| \downarrow 0$ at the central path as $\zeta \rightarrow 1^-$. In the following, we show numerical experiments applying the system (\ref{eq:NGD_trajectory}) with appropriate constant step sizes. Note, by using the logarithmic barrier function formulation for multiple constraints in the system (\ref{eq:NGD_trajectory}), we may still suffer from potential poor scaling behavior near feasible set boundaries.

\subsection{Experiments with common benchmarks}

We first show numerical experiments for the inequality constrained problems in the EA competition at the 2006 IEEE Congress on Evolutionary Computation \cite{liang2006problem}. These benchmarks are widely used among the community of evolutionary algorithms. We choose them as test examples for three reasons. First, they are well defined constrained optimization problems and have different characteristics \cite{rao2016jaya}. Second, they are nontrivial to solve with first-order methods. Third, the relatively simple formulation of the optimization problems allows us to gain deeper insight into the numerical behavior of the present method. We choose the inequality constrained optimization problems for the experimentation. The problem $G12$ is excluded because it has a feasible set consisting of $9^3$ disjointed spheres. For the problem $G24$, which has a feasible set consisting of two disconnected sub-regions, we choose initial designs in the sub-region that contains the reported optimal solution.

We conduct numerical experiments with the parameter $\zeta = 0.98$ and tuned fixed-step sizes. All bound constraints are treated as inequalities. Random initializations, that are away from the reported optimal solution, are selected in the feasible sets.
As shown in table \ref{tb:stepsize_0.98}, apart from the problem G19, we find solutions for the test problems that have an absolute error less than 2e-2 compared to the reported optima. More accurate results can be obtained when we choose shorter step sizes and larger parameter $\zeta$.

For the problem G19, the reported optimal solution $ \mathbf{x}^\star = $ (1.6699e-17, 3.9637e-16, 3.9459, 1.060e-16, 3.2831, 9.9999, 1.1283e-17, 1.2026e-17, 2.5071e-15, 2.2462e-15, 0.3708, 0.2785, 0.5238, 0.3886, 0.2982) is not achievable with a fixed-step size that has a Euclidean norm of $5e-2$. 

For the problems G01, G04, G06, G07, G08, G24, we find sub-optimal solutions that are close to the reported optimal designs. For the problems G09, G10, G18, G19, sub-optimal solutions close to local minima are found. It appears that the ``no free lunch theorem" \cite{wolpert1997no} may apply to the present method implementation: while we solve some of the test problems in no more than a few hundred of iterations, a much larger number of iterations are needed for the remaining problems.

\begin{table} [h]
	\begin{center}
	\caption{Results with the parameter $\zeta = 0.98$ and tuned step sizes}
	\captionsetup{justification=centering}
	\begin{tabular}{|c | c | c | c | c | c |}
		\hline
		Prob. &  step size&  iters. & f($\mathbf{x}^\star$) & f($\mathbf{x_\zeta}$) & abs. error   \\ [0.5ex]
		\hline\hline
		G01  &  0.002 & 2362 & -15 & -14.7215 & 1.86e-2\\  
		\hline
		G04  &  0.2   & 136  & -3.0665e+4 & -3.0657e+4  & 2.85e-4  \\
		\hline
		G06 &  0.002 & 4826 & -6.9618e+3 & -6.8371e+3 & 1.79e-2\\
		\hline
		G07 &  0.0027 & 3009 & 24.3062 & 24.7876 & 1.98e-2 \\
		\hline
		G08 &  0.01 & 66 & -9.5825e-2& -9.5063e-2 & 0.80e-2\\
		\hline
		G09 &  0.05 & 120 & 6.8063e+2 & 6.9238e+2 & 1.73e-2\\
		\hline		
		G10 &  0.35 & 5319 & 7.0492e+3 & 7.1898e+3 & 1.99e-2 \\
		\hline	
		G18 &  0.01 & 257 & -0.8660 & -0.8546 & 1.32e-2 \\
		\hline	
		G19 &  0.05 & 294 & 32.6556 & 2.7120e+2 & 7.30 \\
		\hline	
		G24 &  0.02 & 268 & -5.5080 &  -5.4147 & 1.69e-2 \\
		\hline			
	\end{tabular}
	\label{tb:stepsize_0.98}	
	\end{center}
\end{table}

Still, to get more insight into the numerical behavior of the method implementation, we plot the centrality measure $\cos \theta$ over the optimization process for each test problem in figure \ref{fig:convergence_plots}. The dashed-lines indicate the $\zeta$-neighborhood.
The results may be summarized in three categories:

\textit{Category 1.} G04 and G09: optimization traverses within the $\zeta$-neighborhood;

\textit{Category 2.} G01, G07, G08, G19, G19, and G24: optimization traverses to the $\zeta$-neighborhood but zigzags when it gets close to the solutions;

\textit{Category 3.} G06 and G10: optimization zigzags around the $\zeta$-neighborhood during the optimization process.

The reasons for the zigzagging may be two-folds. First, it may be due to the poor scaling of the logarithmic barrier function near the boundary of the feasible set. The second reason may be the ``overshooting": a large fixed-step size is unable to achieve a small $\zeta$-neighborhood when close to an optimal solution. In problem G06 and G10, the central paths locate closely to the boundaries of the respective feasible sets, thus resulting in zigzagging throughout the whole optimization process. In figure \ref{fig:G06_feasibleset}, we show the narrow feasible set of the test problem G06. The central path traverses close to the two boundaries of the feasible set. A similar phenomenon can be observed in problem G10. In problem G08, the zigzagging disappears when sufficiently short step sizes are chosen. It thus supports our argument of the ``overshooting".

\begin{figure} [h]
	\centering
	\begin{subfigure}{0.3\textwidth}
		\centering\includegraphics[width=\linewidth]{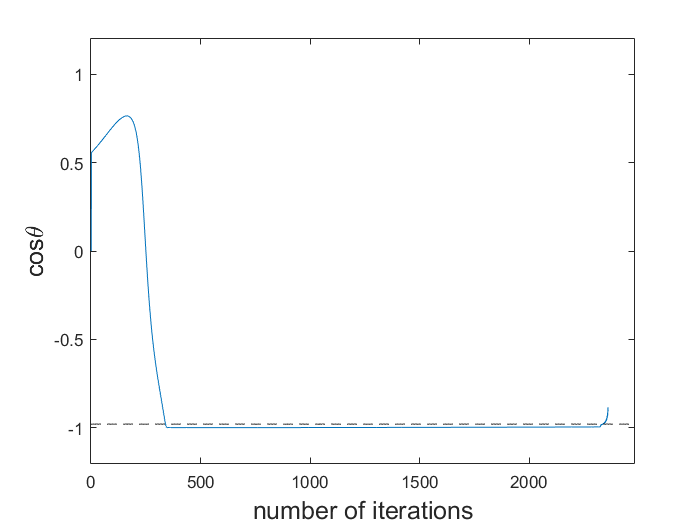}
		\caption{G01}
	\end{subfigure}
	\begin{subfigure}{0.3\textwidth}
		\centering\includegraphics[width=\linewidth]{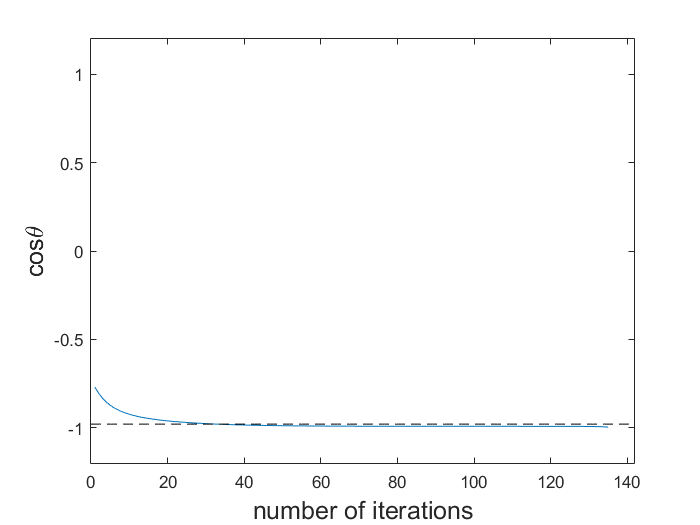}
		\caption{G04}
	\end{subfigure}
	\begin{subfigure}{0.3\textwidth}
		\centering\includegraphics[width=\linewidth]{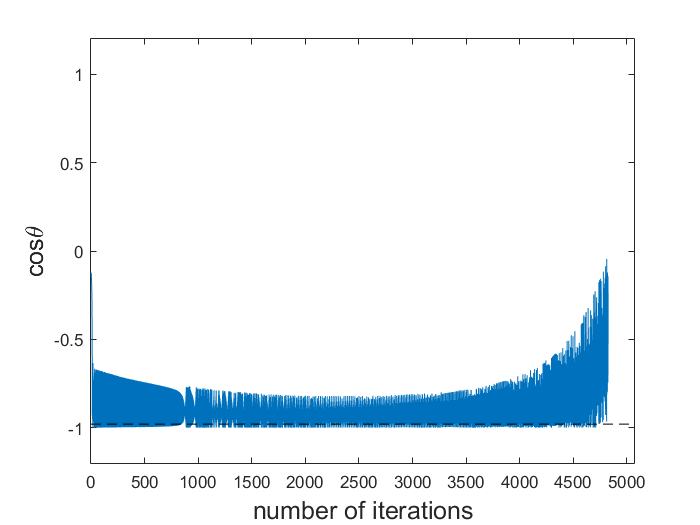}
		\caption{G06}
	\end{subfigure}
	\begin{subfigure}{0.3\textwidth}
		\centering\includegraphics[width=\linewidth]{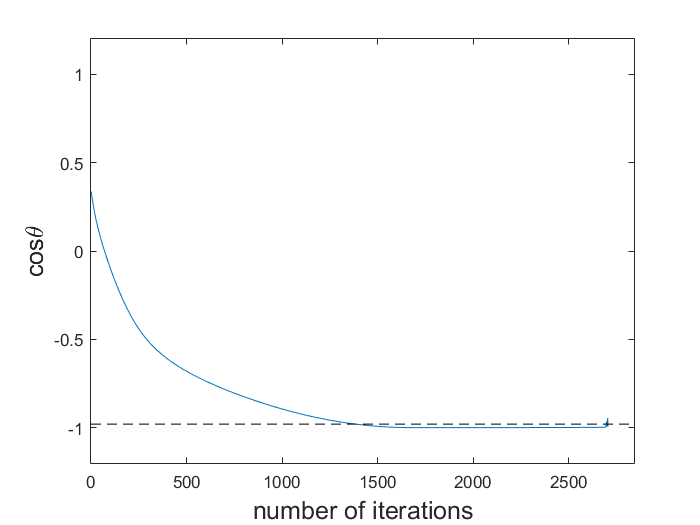}
		\caption{G07}		
	\end{subfigure}
	\begin{subfigure}{0.3\textwidth}
		\centering\includegraphics[width=\linewidth]{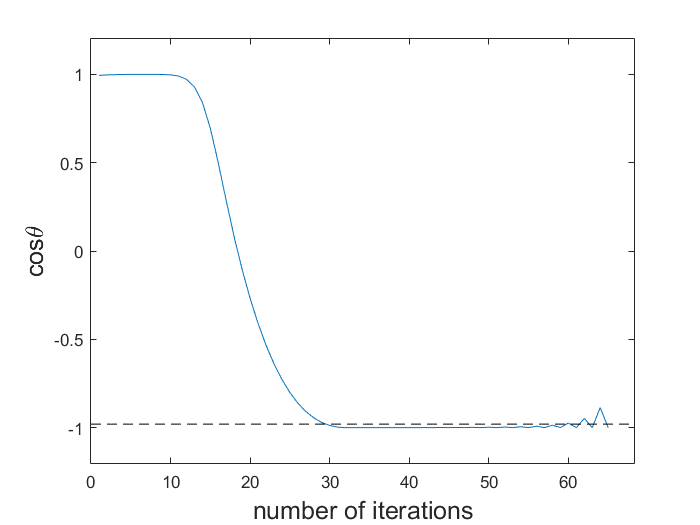}
		\caption{G08}
	\end{subfigure}	
	\begin{subfigure}{0.3\textwidth}
		\centering\includegraphics[width=\linewidth]{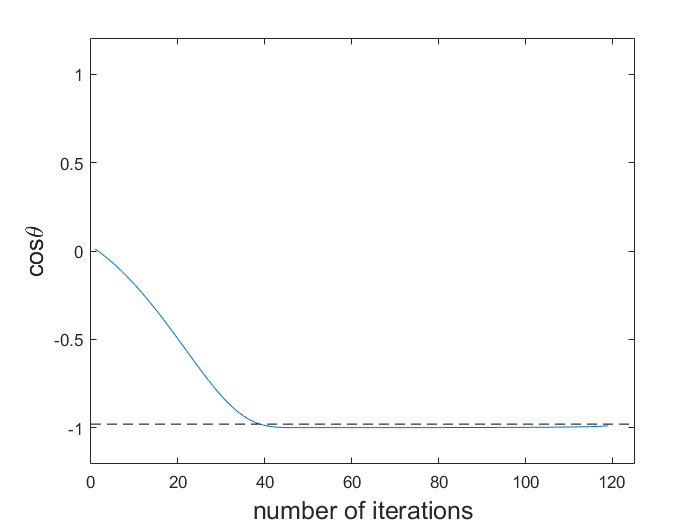}
		\caption{G09}
	\end{subfigure}	
	\begin{subfigure}{0.3\textwidth}
		\centering\includegraphics[width=\linewidth]{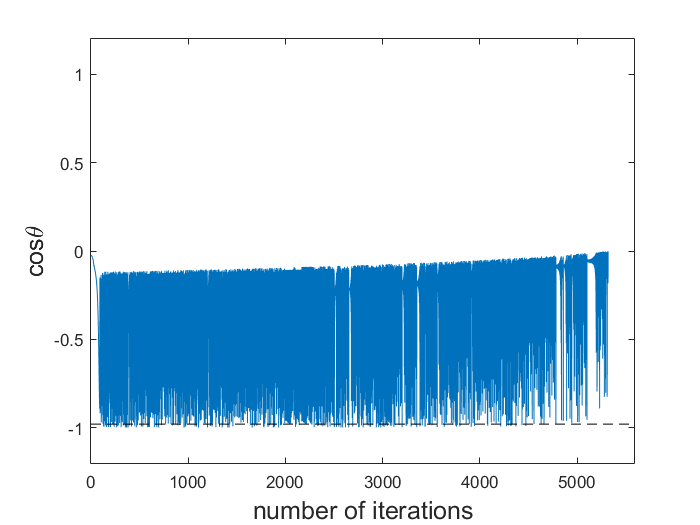}
		\caption{G10}
	\end{subfigure}
	\begin{subfigure}{0.3\textwidth}
		\centering\includegraphics[width=\linewidth]{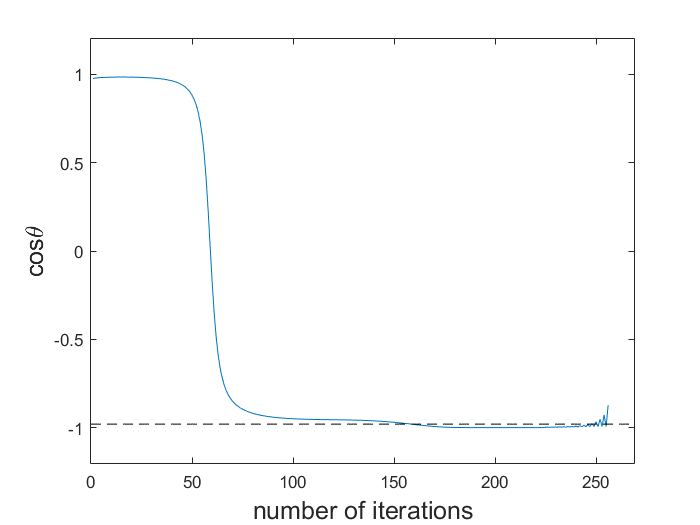}
		\caption{G18}		
	\end{subfigure}
	\begin{subfigure}{0.3\textwidth}
		\centering\includegraphics[width=\linewidth]{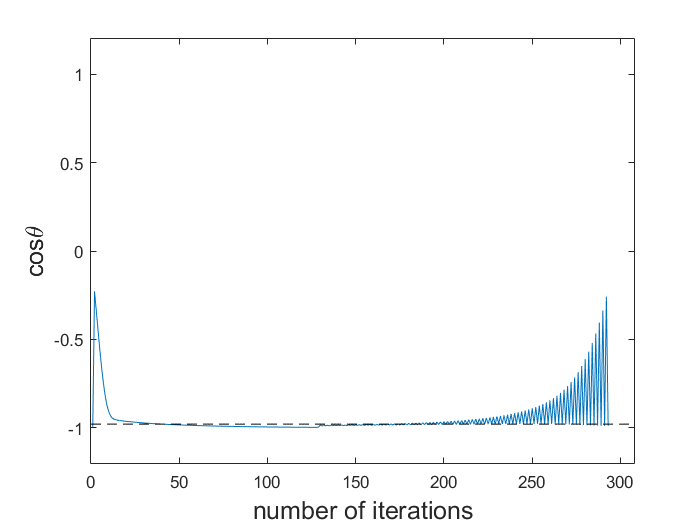}
		\caption{G19}
	\end{subfigure}	
	\begin{subfigure}{0.3\textwidth}
		\centering\includegraphics[width=\linewidth]{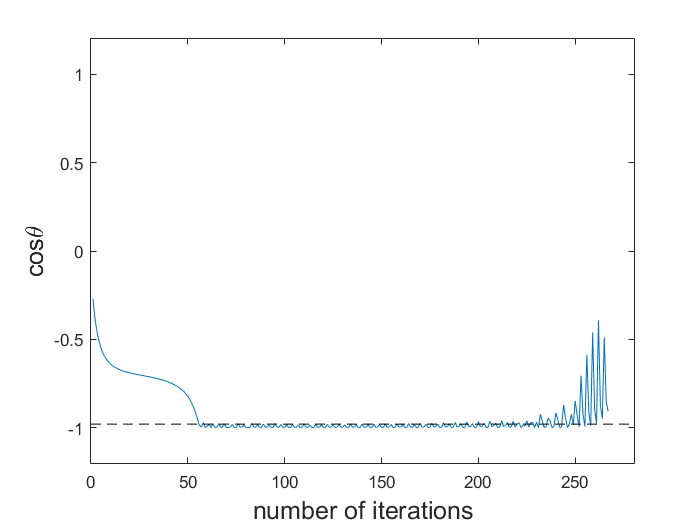}
		\caption{G24}
	\end{subfigure}	
	\caption{$\cos \theta$ plots}
	\label{fig:convergence_plots}
\end{figure}

\begin{figure}[h]
	\centering
	\includegraphics[width=80mm]{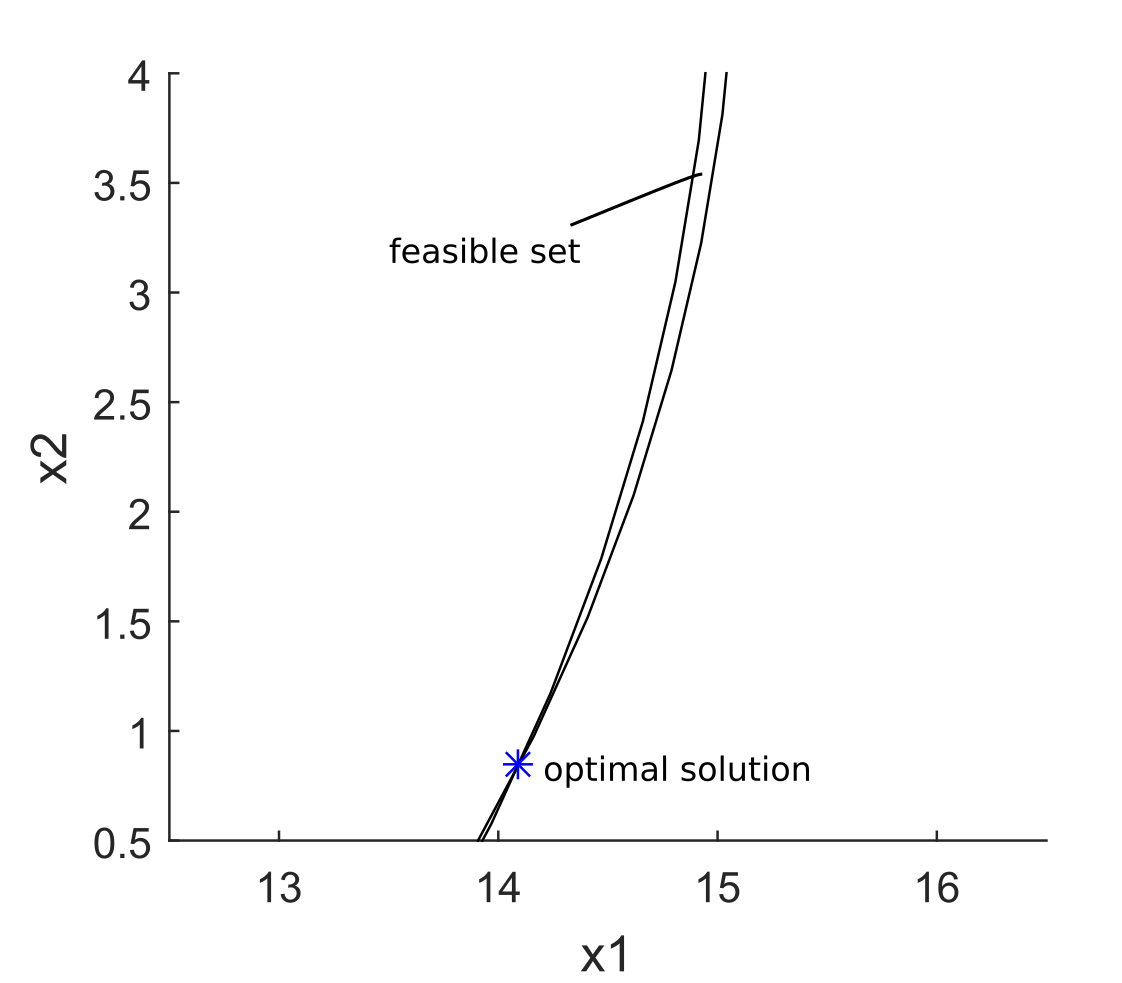}
	\caption{Feasible set for problem G06}
	\label{fig:G06_feasibleset}
\end{figure}

\subsection{Test example in shape optimization}

Although it is shown that the potential poor scaling of the logarithmic barrier function may result in increasing computational effort, we want to point out that, for the problems of \textit{categories 1 and 2}, the logarithmic barrier function formulation can be very efficient in treating a large number of constraints. We support our argument with a shape optimization problem.

Here, we show an academic convex problem. The objective is to maximize the volume of a small sphere. This small sphere is located inside a bigger sphere that acts as the geometric constraint. The shape of both spheres are represented with finite element meshes \cite{zienkiewicz1977finite}. The optimization problem writes
\begin{equation}
	\begin{split}
		& \textnormal{minimize} ~~~   -V(x), \\
		& \textnormal{subject to} ~~~  g_i(x) \leq 0, ~ i = 1,...,m,\\
	\end{split}
\end{equation}
where $V(x)$ is the volume function, $g_i(x)$ is a point-wise defined geometric constraint for the $i$-th design node, $m$ is the number of nodes of the design mesh, and $x \in \mathbb{R}^{3m}$ is the field of nodal coordinates of the design sphere mesh. The number of nodes of the small sphere (design sphere) is 19897. Thus, the total number of design variables is 59691 and the total number of constraints is 19897. We use the logarithmic barrier function for multiple constraints and choose the parameter $\zeta = 0.95$. In figure \ref{fig:sphere_iters}, the shape variation process with depicted iterations is shown. Initially, the design sphere is located close to the boundary of the constraint sphere. During the shape variation process, it moves towards the center of the constraint sphere, while adapting its shape at each iteration. One could recognize easily that the central path of this optimization problem is being approached and followed until the solution is found when constraints become active.


\begin{figure}[h]
	\centering 
	\begin{subfigure}{0.24\textwidth}
		\centering\includegraphics[width=\linewidth]{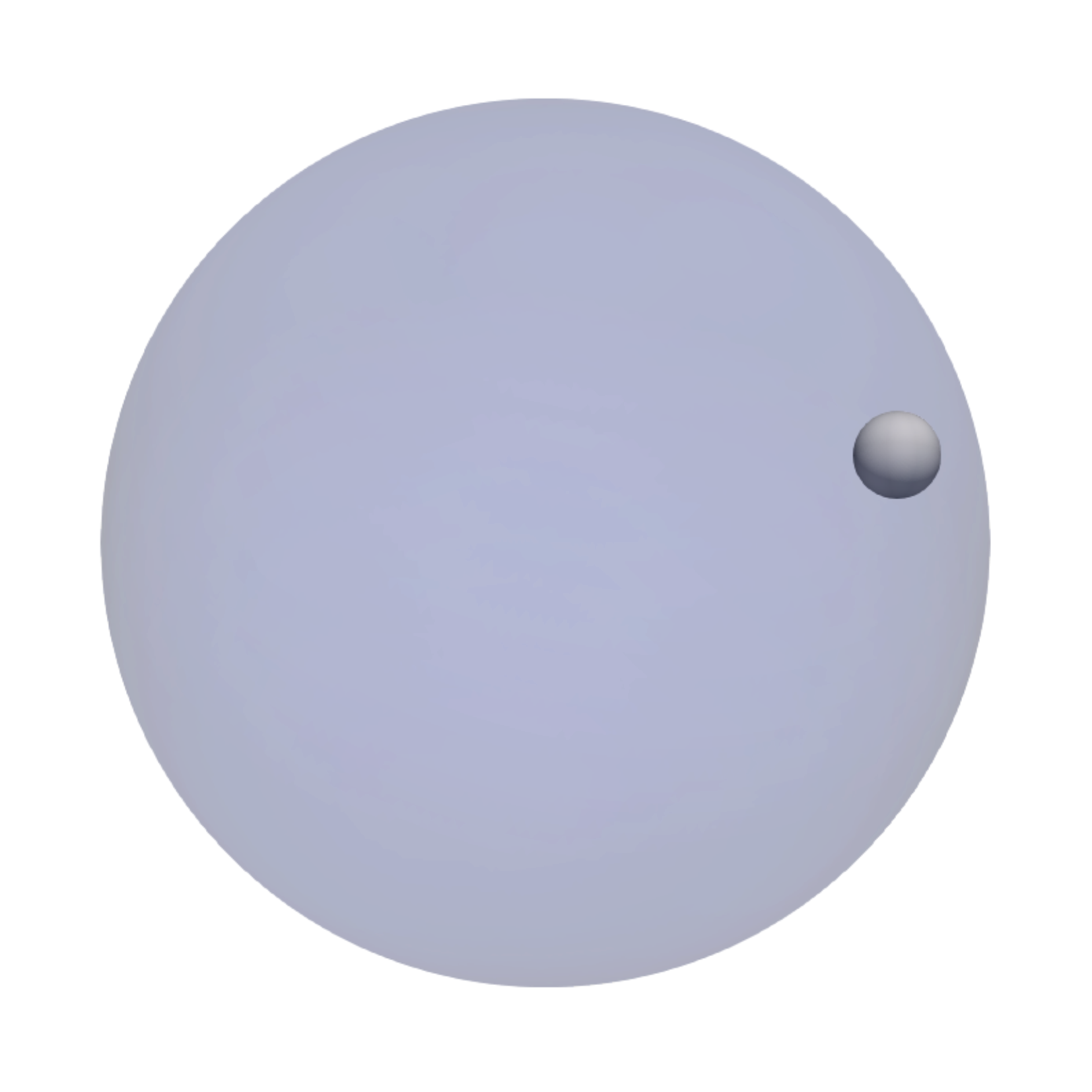}
		\caption{Initial design}
		\label{fig:1}
	\end{subfigure}\hfil 
	\begin{subfigure}{0.24\textwidth}
		\centering\includegraphics[width=\linewidth]{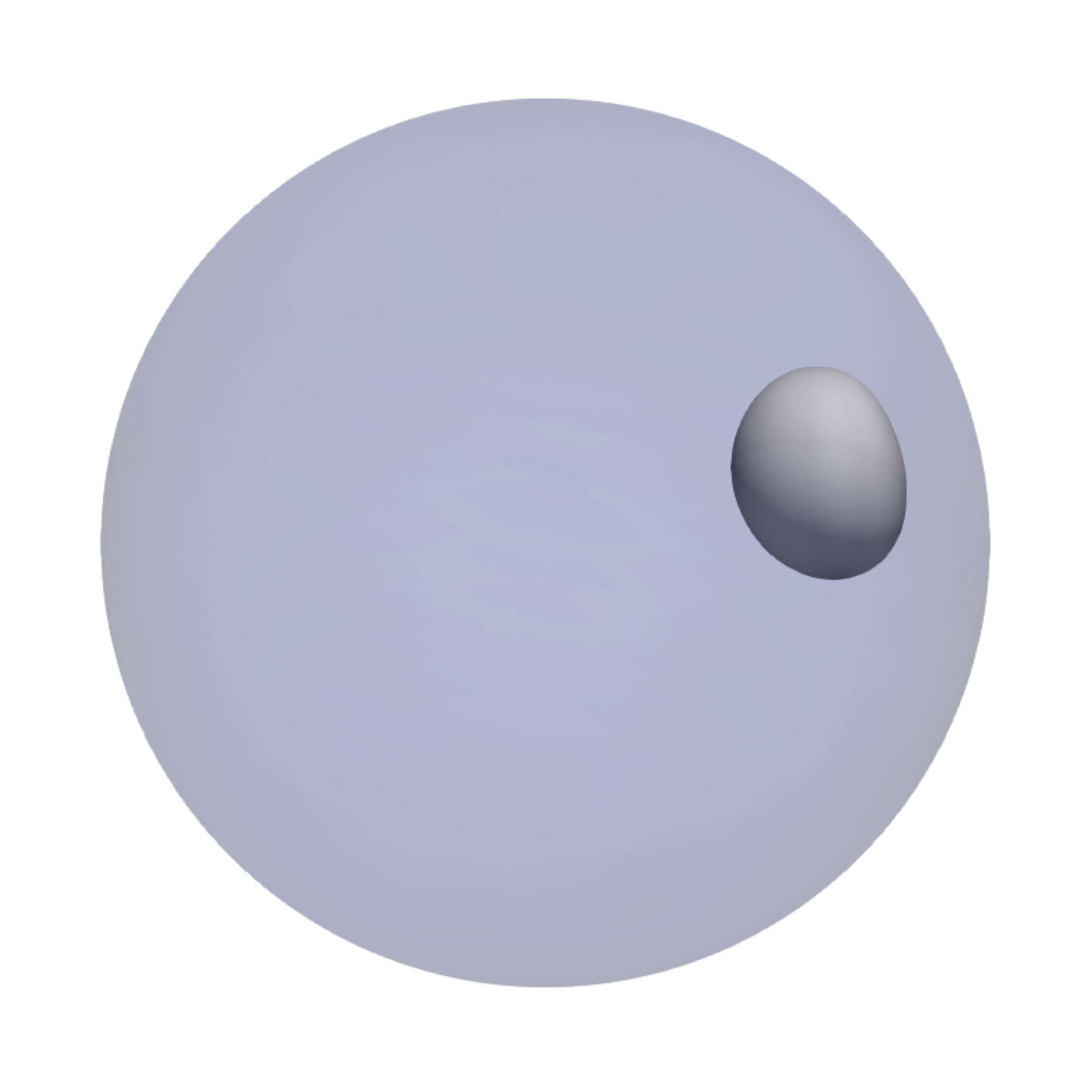}
		\caption{Iteration 15}
		\label{fig:2}
	\end{subfigure}\hfil 
	\begin{subfigure}{0.24\textwidth}
		\centering\includegraphics[width=\linewidth]{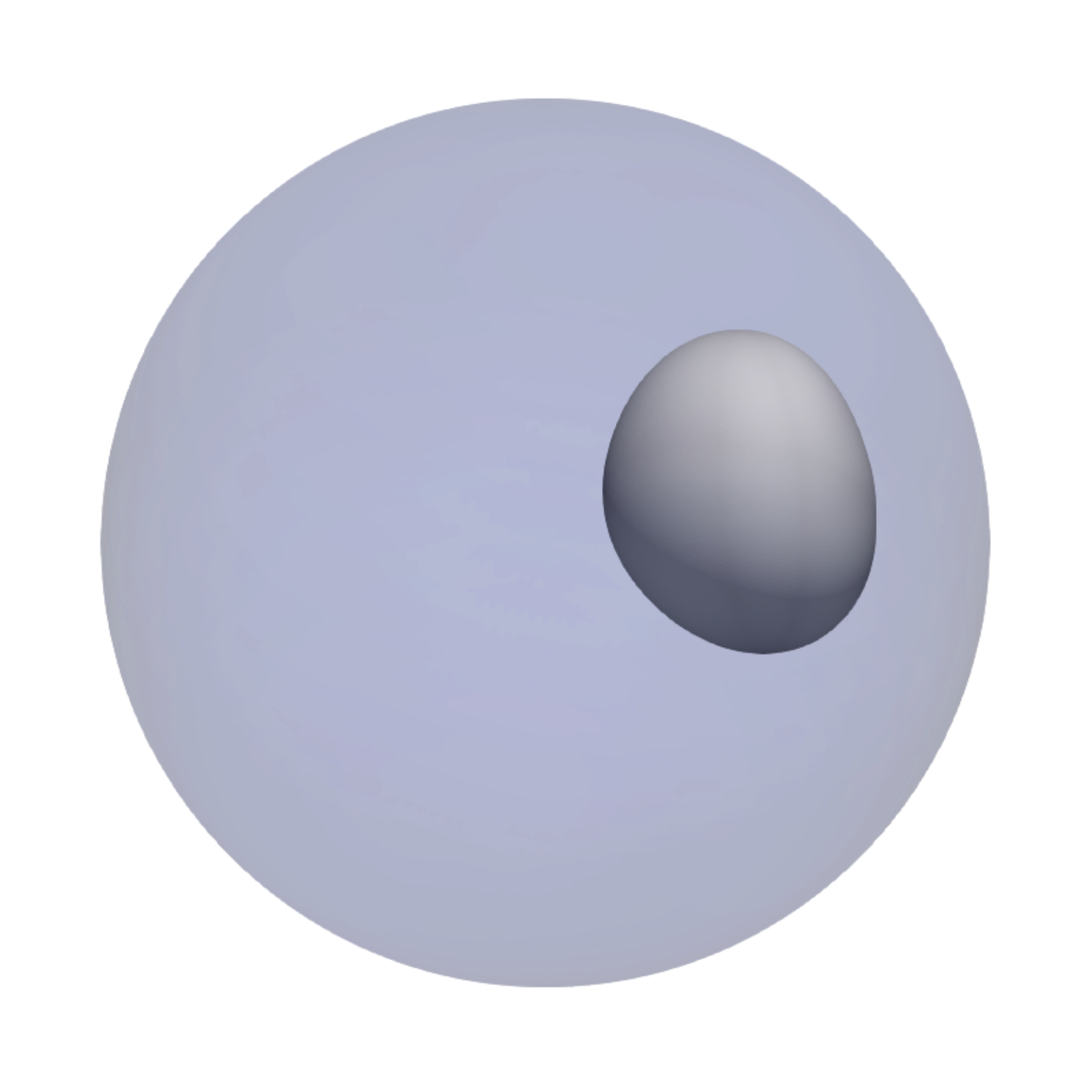}
		\caption{Iteration 30}
		\label{fig:3}
	\end{subfigure}	
	\begin{subfigure}{0.24\textwidth}
		\centering\includegraphics[width=\linewidth]{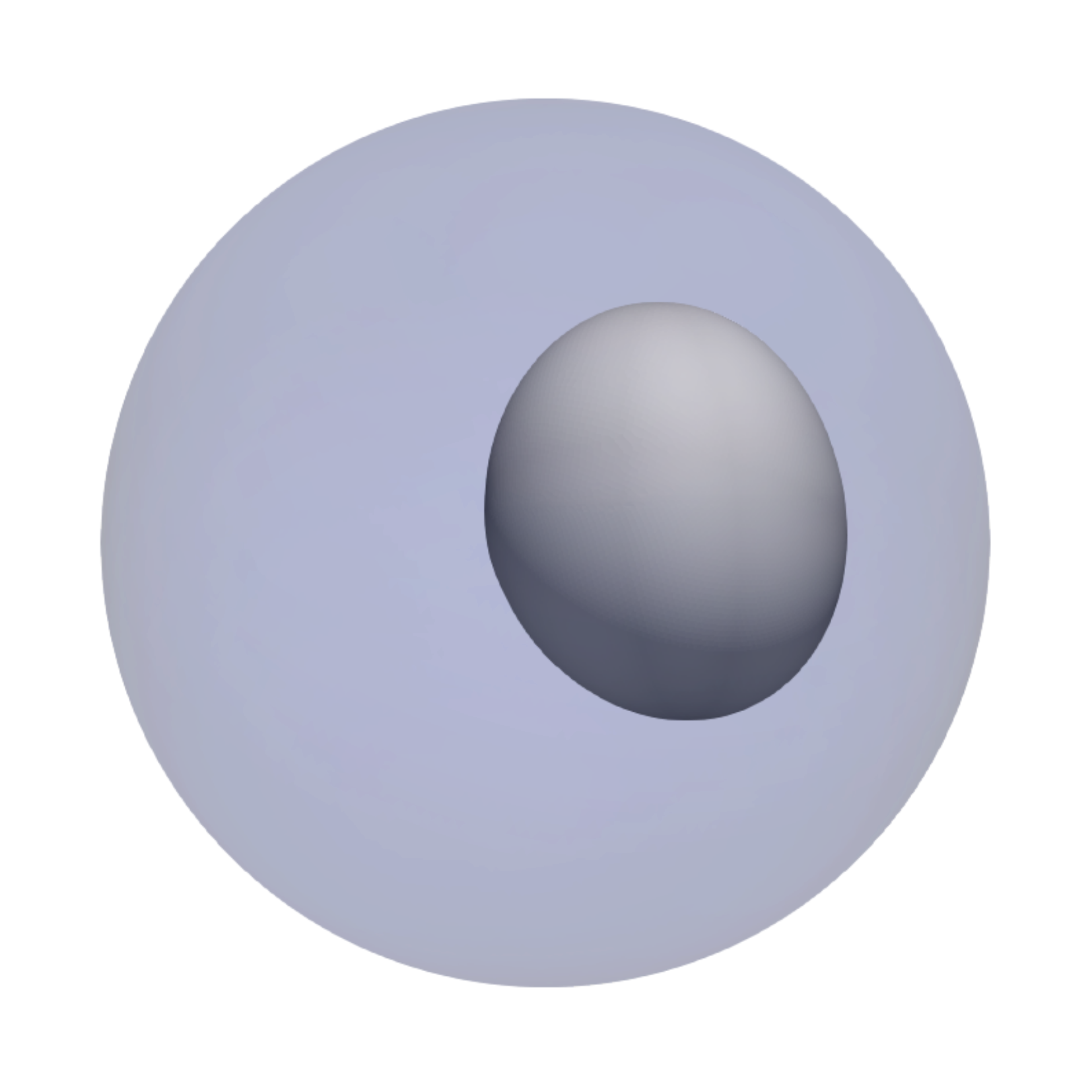}
		\caption{Iteration 45}
		\label{fig:4}
	\end{subfigure}		
	
	\medskip
	\begin{subfigure}{0.24\textwidth}
		\centering\includegraphics[width=\linewidth]{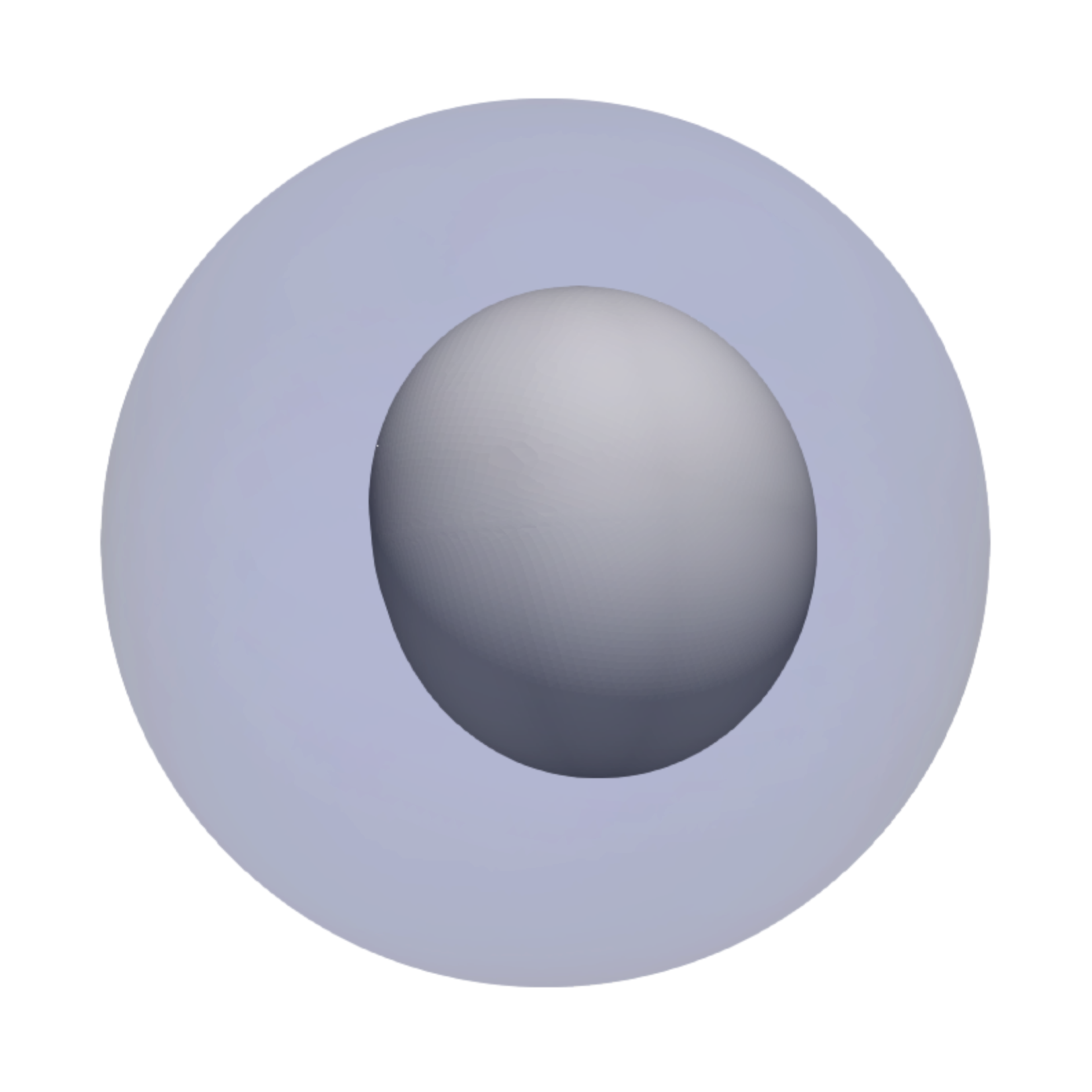}
		\caption{Iteration 60}
		\label{fig:5}
	\end{subfigure}\hfil 
	\begin{subfigure}{0.24\textwidth}
		\centering\includegraphics[width=\linewidth]{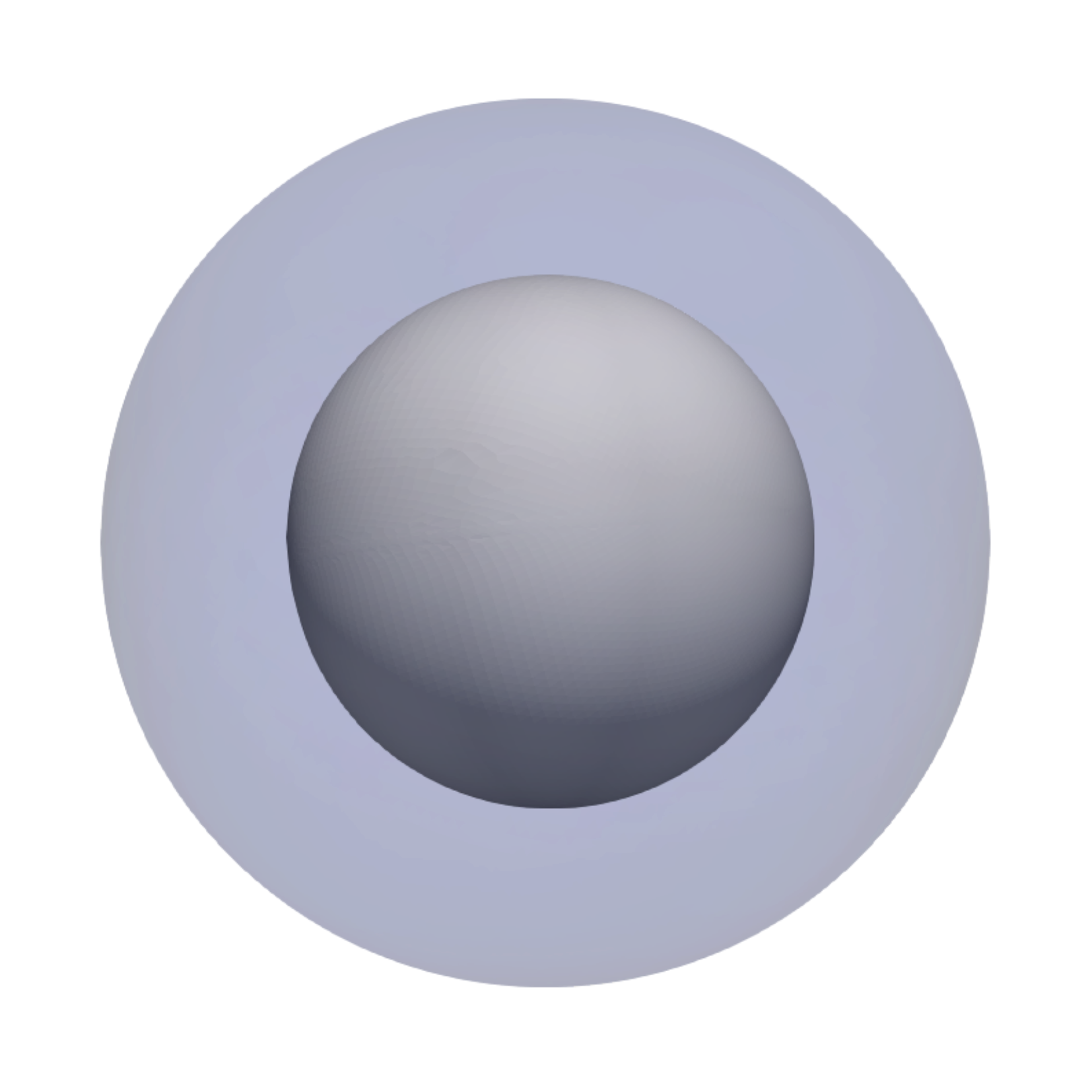}
		\caption{Iteration 70}
		\label{fig:6}
	\end{subfigure}\hfil 
	\begin{subfigure}{0.24\textwidth}
		\centering \includegraphics[width=\linewidth]{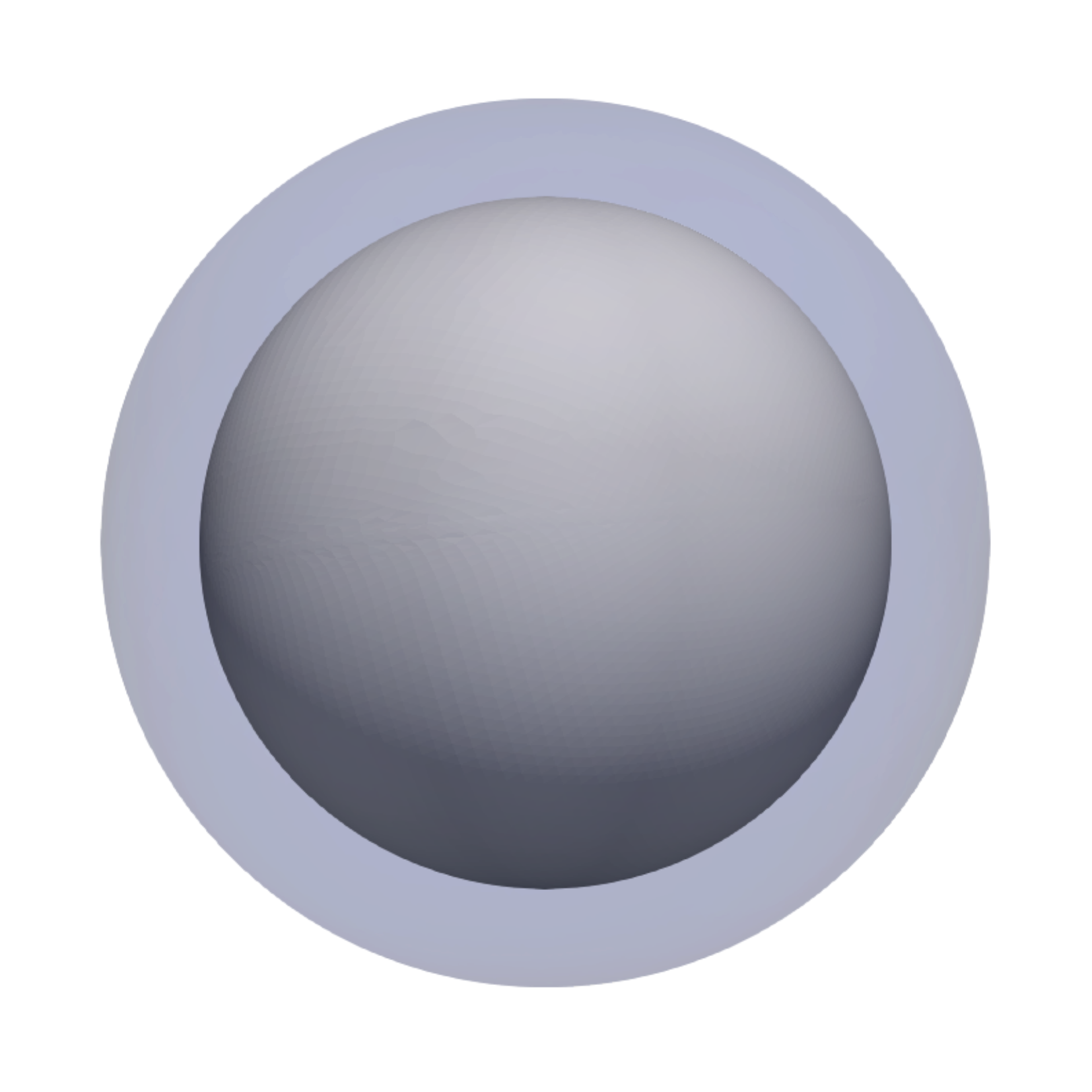}
		\caption{Iteration 80}
		\label{fig:7}
	\end{subfigure}
	\begin{subfigure}{0.24\textwidth}
		\centering	\includegraphics[width=\linewidth]{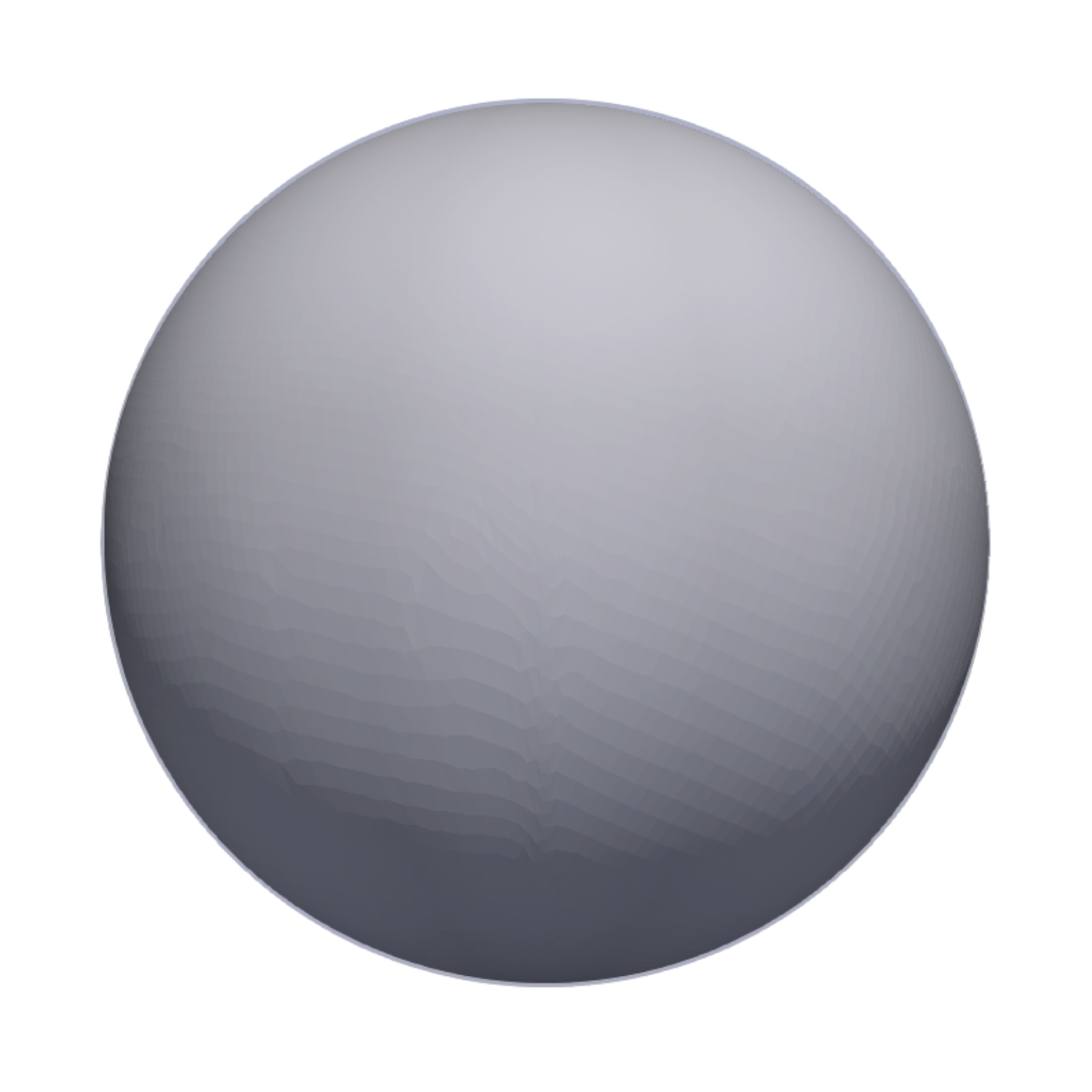}
		\caption{Iteration 101}
		\label{fig:8}
	\end{subfigure}	
	\caption{Design updates}
	\label{fig:sphere_iters}
\end{figure}


\subsection{Real-world application to shape optimization}
We consider a real-world application to shape optimization. The present method is implemented in ShapeModule, which is a flexible solver-agnostic optimization platform and provides optimization algorithms as well as shape control methods, such as Vertex Morphing \cite{Hojjat}. The optimization problem is to minimize the mass of a frame structure under load-displacement constraint (i.e., the displacement of every surface node is bounded). The optimization problem writes
\begin{equation}
	\begin{split}
		& \textnormal{minimize} ~~~~   M(x), \\
		& \textnormal{subject to} ~~~  g_i(x) \leq 0, ~ i = 1,...,m,
	\end{split}
\end{equation}
where $M(x)$ is the function for the mass, $g_i(x)$ is a point-wise formulated displacement constraint for the $i$-th node, $m$ is the number of nodes of the design surface mesh, and $x \in \mathbb{R}^{3m}$ is the field of nodal coordinates of the design surface mesh.  The number of design variables is 144423, and the number of constraints is 48141. Note that for multiple constraints, we can use the logarithmic barrier function formulation as in the previous test examples. Each single displacement constraint gradient can be efficiently computed using the adjoint sensitivity analysis. In this application, we use the load-displacement sensitivity provided by the software OptiStruct to conform with a standard industrial design chain. We choose the parameter $\zeta = 0.95$.

\begin{figure}[h]
	\centering
	\begin{subfigure}{8cm}
		\centering\includegraphics[width=80mm]{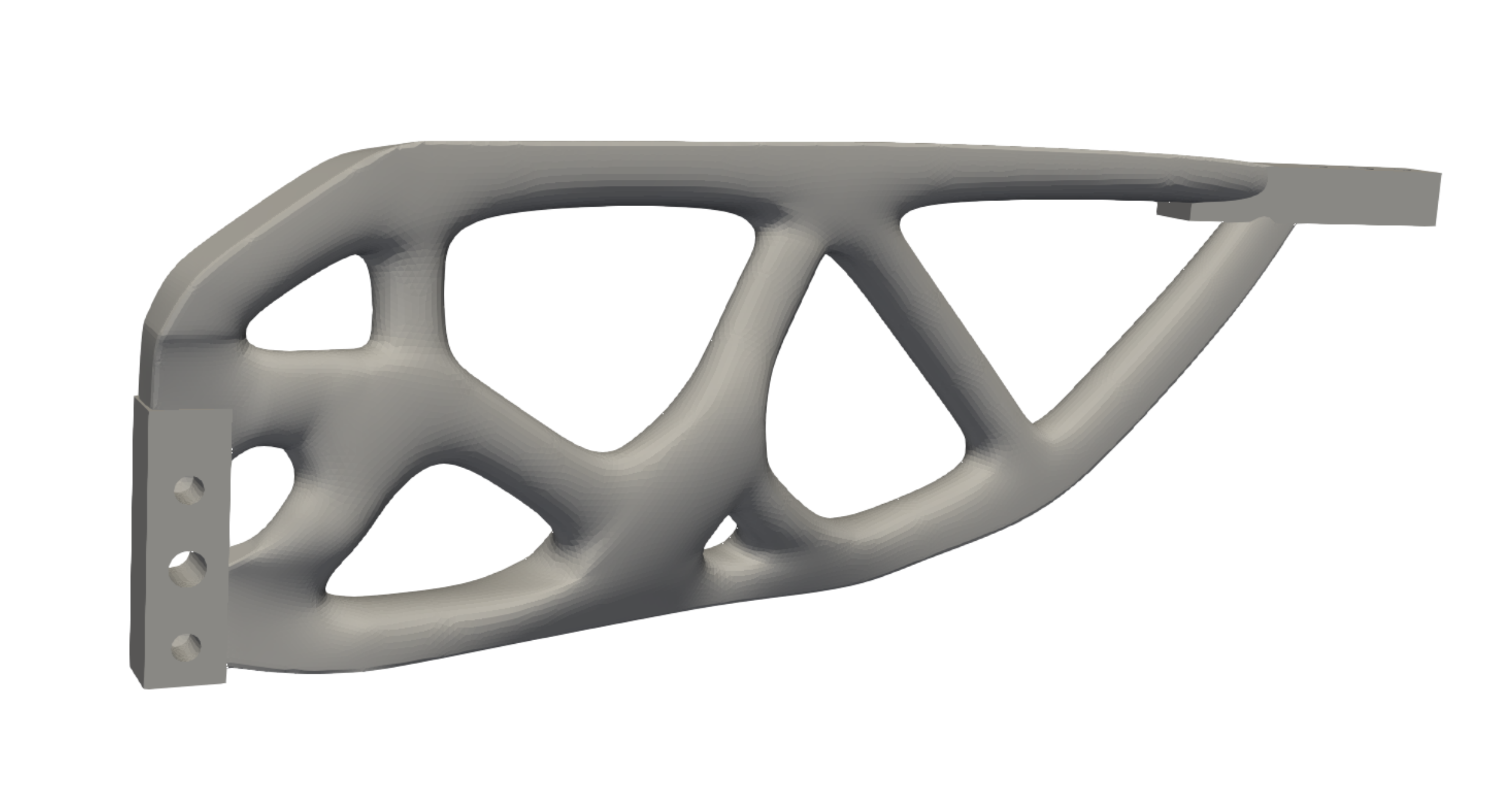}
		\caption{The initial frame design}
		\label{fig:frame_iter_1}
	\end{subfigure}
	\begin{subfigure}{8cm}
		\centering\includegraphics[width=80mm]{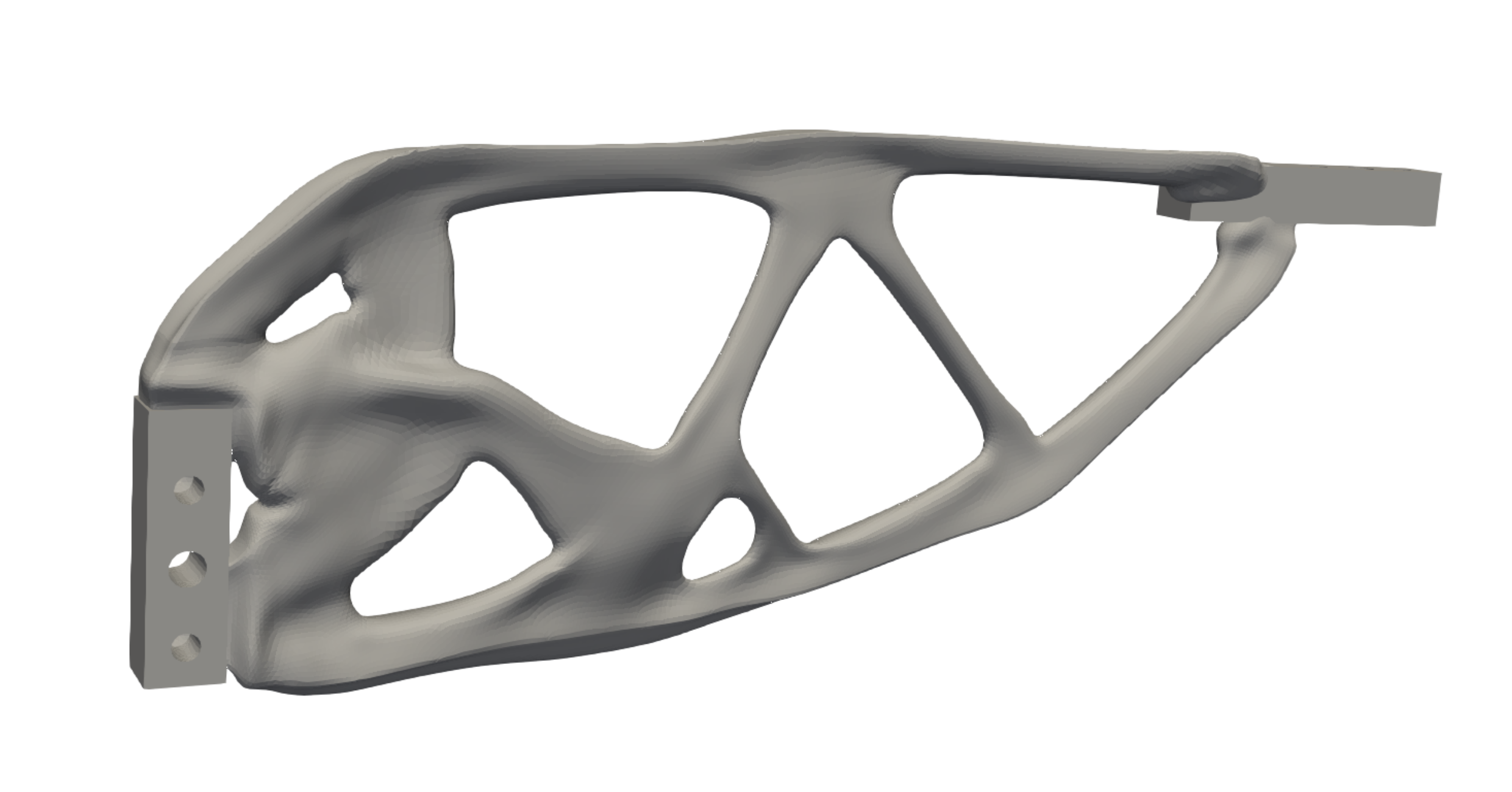}
		\caption{The optimized frame design}
		\label{fig:frame_iter_193}
	\end{subfigure}
	\caption{Design optimization of a real-world frame structure}
	\label{fig:frame}
\end{figure}

\begin{figure}[h]
	\centering
	\begin{footnotesize}
	\executeiffilenewer{fig/frame_objective1.svg}{fig/frame_objective1.pdf}%
	{inkscape -z -C --file=fig/frame_objective1.svg %
		--export-pdf=fig/frame_objective1.pdf --export-latex}%
	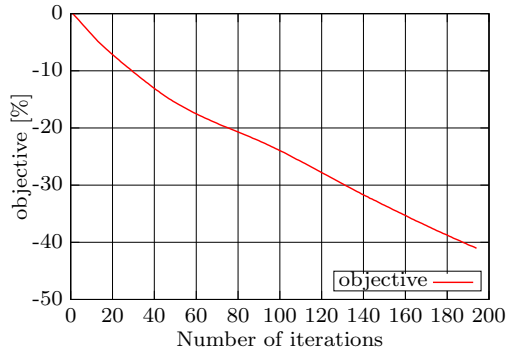%

		\caption{Plot of the frame objective}		
		\label{fig:frame_objective}
	\end{footnotesize}
\end{figure}

\begin{figure}[h]
	\centering
	\begin{footnotesize}
	\executeiffilenewer{fig/frame_constraint_plot1.svg}{fig/frame_constraint_plot1.pdf}%
	{inkscape -z -C --file=fig/frame_constraint_plot1.svg %
		--export-pdf=fig/frame_constraint_plot1.pdf --export-latex}%
	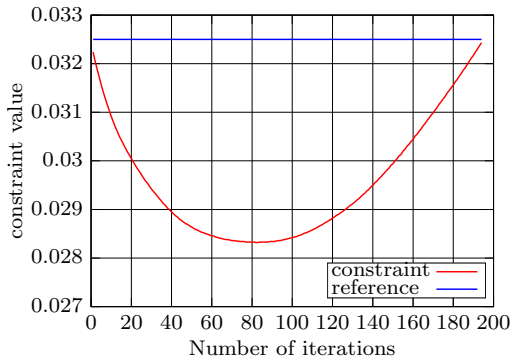%

		\caption{Plot of the frame constraint}		
		\label{fig:frame_constraint_plot}
	\end{footnotesize}
\end{figure}

\begin{figure}[h]
	\centering
	\begin{footnotesize}
	\executeiffilenewer{fig/frame_centrality1.svg}{fig/frame_centrality1.pdf}%
	{inkscape -z -C --file=fig/frame_centrality1.svg %
		--export-pdf=fig/frame_centrality1.pdf --export-latex}%
	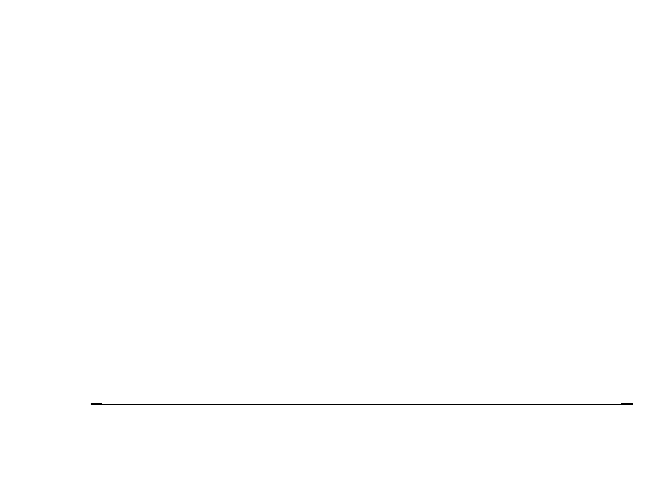%

		\caption{Plot of the centrality measure}
		\label{fig:frame_centrality}
	\end{footnotesize}
\end{figure}

In figure \ref{fig:frame} we show the initial frame design and the shape optimized design after 194 iterations. The mass of the structure is reduced by $ 41\% $ as shown in figure \ref{fig:frame_objective}. In figure \ref{fig:frame_constraint_plot}, we show a plot of maximum constraint value $g = max\{g_i\}$ of each iteration. In figure \ref{fig:frame_centrality}, we show that the optimization is able to approach and follow a central path within the $\zeta$-neighborhood.

\begin{myremark}
	By following a central path, an intermediate design improves not only the objective function but also the constraint function. Take the design of iteration 80 as an example: the mass is reduced by $20.5\%$, and the displacement is reduced by $12.1 \%$. These designs may enrich the design options if the original problem is reformulated as a bi-objective optimization problem, in which both mass and the maximum displacement are set as objectives. The resulting intermediate designs alongside a central path are approximated Pareto solutions.
\end{myremark}

\section{Conclusion}
\label{sec:conclusion}
This paper proposes a gradient descent akin method for solving inequality constrained nonlinear programming problems. We show the global behavior and convergence of the method and prove local convergence under the introduced \textit{relative convex condition}. Robustness is shown in various computational test problems. It is also shown that the method exhibits the potential for solving nonconvex nonlinear problems. The method may be especially suited for large-scale problems due to the very cheap computational cost in each iteration. Practical implementations based on the present method are of interest in future work.

\section*{Acknowledgments.}
The work by the author LC was done during his work as one of the coordinators at the Bavarian Graduate School of Computational Engineering (BGCE). The working experiences and financial support are gratefully acknowledged. The authors are grateful to the ShapeModule team at the BMW Group for providing a real-world model and their framework. The authors also thank Jian Cui from the Helmholtz Pioneer Campus for his generosity in proofreading this manuscript.

\bibliographystyle{spmpsci}      
\bibliography{references}


\end{document}